\documentclass[final,leqno]{siamltex}
\usepackage{amsmath}
\usepackage{amssymb}
\usepackage{graphicx}
\usepackage{color}
\usepackage{ifpdf}
\usepackage{url}
\usepackage{hyperref}

\usepackage{booktabs}
\usepackage{empheq}
\usepackage[ruled,vlined]{algorithm2e}
\usepackage{algorithmicx}
\usepackage{algpseudocode}
\usepackage{multirow}
\usepackage{enumitem}

\newtheorem{assumption}{Assumption A.}

\newcommand{\R}{\mathbb{R}}

\newcommand{\abs}[1]{\left\vert#1\right\vert}
\newcommand{\set}[1]{\left\{#1\right\}}
\newcommand{\norm}[1]{\left\Vert#1\right\Vert}

\newtheorem{remark}{Remark}

\newcommand{\Eproof}{\hfill$\square$}

\newcommand{\xb}{\mathbf{x}}
\newcommand{\yb}{\mathbf{y}}

\newcommand{\ub}{\mathbf{u}}
\newcommand{\vb}{\mathbf{v}}
\renewcommand{\sb}{\mathbf{s}}
\newcommand{\db}{\mathbf{d}}

\newcommand{\Xb}{\mathbf{X}}

\newcommand{\dom}[1]{\mathrm{dom}\left(#1\right)}
\newcommand{\trace}[1]{\mathrm{tr}\left(#1\right)}
\renewcommand{\vec}[1]{\mathrm{vec}\left(#1\right)}

\DeclareMathOperator*{\argmin}{arg\,min}

\title{An Inexact Proximal Path-Following Algorithm \\ for Constrained  Convex Minimization}

\author{Quoc Tran-Dinh$^{*}\!\!\!$\and Anastasios Kyrillidis$^{*}\!\!\!$\and Volkan Cevher
\thanks{Laboratory for Information and Inference Systems (LIONS), 
	     \'{E}cole Polytechnique F\'{e}d\'{e}rale de Lausanne (EPFL), 
              CH1015 - Lausanne, Switzerland.
\newline
E-mail: {\tt\{quoc.trandinh, anastasios.kyrillidis, volkan.cevher\}@epfl.ch}.}}

\begin{document}
\maketitle

\begin{abstract}
Many scientific and engineering applications feature nonsmooth convex minimization problems over convex sets. 
In this paper, we address an important instance of this broad class where we assume that the nonsmooth objective is equipped with a tractable proximity operator and that the convex constraint set affords a self-concordant barrier. 
We provide a new joint treatment of proximal and self-concordant barrier concepts  and illustrate that such problems can be efficiently solved, without the need of lifting the problem dimensions, as in disciplined convex optimization approach. 
We propose an inexact path-following algorithmic framework and theoretically characterize the worst-case analytical complexity of this framework when the proximal subproblems are solved inexactly. To show the merits of our framework, we apply its instances to both synthetic and real-world applications, where it shows advantages over standard interior point methods. As a by-product, we describe how our framework can obtain points on the Pareto frontier of regularized problems with self-concordant objectives in a tuning free fashion.
\end{abstract}


%%%%%%%%%%%%%%%%%%%%%%%%%%%%%%%%%%%%%%%%%%%%%%%%%%%%%%%%%%%%
%% 1. Introduction.
%%%%%%%%%%%%%%%%%%%%%%%%%%%%%%%%%%%%%%%%%%%%%%%%%%%%%%%%%%%%
\section{Problem statement and motivation}\label{sec:intro}
We consider the following constrained convex minimization problem, which has aplenty applications in diverse disciplines, including  machine learning, signal processing, statistics, and control \cite{Boyd1994, Harmany2012, Hassibi1999, Stadler2012}:
\begin{equation}\label{eq:constr_cvx_prob}
g^{*} := \min_{\xb\in\Omega} g(\xb).
\end{equation}
Here, $\Omega \subseteq \mathbb{R}^n$ is a nonempty, closed and convex set and $g$ is a (possibly) non-smooth convex function from $\mathbb{R}^n\to \mathbb{R}\cup\set{+\infty}$.

Problem \eqref{eq:constr_cvx_prob} is sufficiently generic to cover many optimization settings considered in the literature.
Under mild assumptions on $g$ and $\Omega$, general convex optimization methods such as  mirror descent, projected subgradient as well as Frank-Wolfe methods  can be applied to solve \eqref{eq:constr_cvx_prob}, as long as the projection on $\Omega$ can be computed efficiently \cite{Andreasson2006,Beck2003,Ben-Tal2001,Nesterov2009b}. 
Theoretically, such methods are usually slowly convergent (e.g., $O(1/\sqrt{k})$ global convergence rate, where $k$ is the iteration counter) and sensitive to the choice of step-sizes \cite{Nesterov2004}, which constitutes them impractical for many applications.

In the case  $\Omega$ has an explicit form, e.g., $\Omega := \set{\xb\in\R^n ~|~ c_i(\xb) \leq 0, ~i=1,\cdots, m}$, where $c_i(\cdot)$ are generic convex functions for $i=1,\cdots, m$, nonsmooth optimization methods such as level and bundle methods are also potential candidates to solve \eqref{eq:constr_cvx_prob} \cite{Bonnans1997,Karas2009,Lemarechal1995,Nesterov2004,Sagastizabal2005}. 
Though, similarly to subgradient methods, bundle schemes also have slow global convergence rate (e.g., $\mathcal{O}(1/\sqrt{k})$) when applying to nonsmooth problems, except for the cases where particular assumptions are made \cite{Lan2013a,Nesterov2004}.

When $g$ is a smooth term and projections on $\Omega$ are expensive to compute, sequential convex programming approach such as sequential quadratic programming (SQP) constitutes an efficient strategy for solving \eqref{eq:constr_cvx_prob}, see, e.g., \cite{Andreasson2006,Fletcher1987,Nocedal2006}. 
However, this approach is usually generic and still requires a globalization strategy to ensure convergence.
Along this line, the most famous class of algorithms is that of interior point methods (IPM) that solve standard conic programming problems in polynomial time \cite{BenTal2001,Nesterov1994}. The key structure exploited in conventional IPMs is the existence of a \emph{barrier function} for the feasible set $\Omega$ (cf., Section \ref{sec:background}).

In such cases, one considers the penalized family of parametric composite convex optimization problems:
\begin{equation}\label{eq:cvx_prob}
\min_{\xb \in \mathbf{R}^n} \set{ F(\xb; t) :=  f( \xb ) + t^{-1}g(\xb) },  
\end{equation} 
where $t > 0$ is a penalty parameter and $f$ is the barrier function over the set $\Omega$. By solving \eqref{eq:cvx_prob} for a sequence of decreasing $t$ values, i.e.,  $t \downarrow 0^{+}$, we can trace the analytic central path $\xb^{*}_t$ of \eqref{eq:cvx_prob} as it converges to the solution $\xb^{\ast}$ of \eqref{eq:constr_cvx_prob}. 

However, assuming no further structure in $f$, the resulting path-following scheme is not guaranteed to converge, and solving \eqref{eq:cvx_prob} becomes harder as $t \downarrow 0^{+}$; see, e.g., \cite{Nemirovski2008}. Aptly, Nesterov and Nemirovskii \cite{Nesterov2004} introduced the {\it self-concordance} concept (cf., Section \ref{sec:background} for definitions), which characterizes a broad collection of penalty functions $f$ and guarantees the polynomial-solvability of \eqref{eq:cvx_prob}, by sequentially using Newton methods. 

Within the IPM context, $g$ is usually assumed to be a smooth term. When $g$ is a \emph{nonsmooth} term, it has a direct impact on the computational effort. Such problems do occur frequently in applications.
Examples include but are not limited to sparse concentration matrix estimation with $\ell_1$-norm (eq. (11) in \cite{Ravikumar2011}), data clustering with $\ell_1$-norm (semidefinite programming reformulation in Section 4.1 of \cite{Jalali2012}), spectral line estimation with atomic norms (eq. (2.6) in \cite{Tang2012} and eq.\ (3.4) in \cite{Bhaskar2012}), etc. Since off-the-shelf IPMs usually approximate $g^{\ast}$ by solving a sequence of {\it smooth} problems \cite{Nemirovski2008,Nesterov2004}, $g$ in \eqref{eq:cvx_prob} must allow a reformulation where standard IP solvers can be applied
(i.e., via disciplined convex optimization (DCO) techniques \cite{Grant2006}).

Nevertheless, the DCO approach can inflate the problem dimensions, and often suffers from the curse-of-dimensionality.  As a concrete example, consider the max-norm clustering problem \cite{Jalali2012}, where we seek a clustering matrix $\mathbf{K}$ that minimizes disagreement with a given affinity matrix $\mathbf{A}$:
\begin{equation}\label{eq:clustering_prob}
\begin{array}{cl}
\displaystyle\min_{\mathbf{L}, \mathbf{R}, \mathbf{K} \in\mathbb{R}^{p\times p}} & \norm{\vec{\mathbf{K} - \mathbf{A}}}_1\\
\textrm{s.t.} & \begin{bmatrix}\mathbf{L} & \mathbf{K} \\ \mathbf{K}^T & \mathbf{R} \end{bmatrix} \succeq 0, ~\mathbf{L}_{ii} \leq 1, ~\mathbf{R}_{ii}\leq 1, ~i=1,\dots, p,
\end{array}
\end{equation}
where $\mathrm{vec}$ is the vectorization operator of a matrix (i.e., $\mathrm{vec}(\Xb) := (\Xb_1^T, \cdots, \Xb_n^T)^T$, where $\Xb_i$ is the $i$-th column of $\Xb$).
This non-smooth formulation affords rigorous theoretical guarantees for its solution quality and can be formulated as a standard conic program. Unfortunately, we need to add $\mathcal{O}(p^2)$ slack variables to process the $\ell_1$-norm term and  the linear constraints. Moreover, the scaling factors (e.g., the Nesterov-Todd scaling factor regarding the semidefinite cone \cite{Nesterov1997}) can create memory bottlenecks by destroying the sparsity of the underlying problem (e.g., by leading to dense KKT matrices in Newton systems). 

%Consequently, the efficiency of conventional IPM's significantly degrade.
\subsection{Our approach}
In general, when the penalty function $f$ has a Lipschitz continuous gradient \cite{Nesterov2004} and $g$ has a computable proximity operator (cf., Section \ref{sec:background} for definitions), several efficient convex optimization algorithms for solving \eqref{eq:cvx_prob} exist \cite{Bauschke2011,Beck2009,Nesterov2004,Nesterov2007}. However, to the best of our knowledge, there has been no unified framework for path-following schemes of \eqref{eq:cvx_prob} where $f$ is a self-concordant barrier (hence, {\it non-globally Lipschitz continuous gradient}) and $g$ is a non-smooth term with \textit{proximal tractability} (i.e., the proximal operator of $g$ is efficient to compute).

To this end, we  address \eqref{eq:constr_cvx_prob} with a new {\it proximal} path-following scheme, which solves \eqref{eq:cvx_prob} for a sequence of {\it adaptively selected} parameters $t_k$. Our scheme guarantees the following: If $\xb_{t_k}$ is an approximation of $\xb^{*}_{t_k}$ of \eqref{eq:constr_cvx_prob}, by solving instances of \eqref{eq:cvx_prob} for $t = t_k$---i.e., within some user-defined accuracy, then our method produces an approximate solution $\xb_{t_{k+1}}$ of $\xb^{*}_{t_{k+1}}$ of \eqref{eq:constr_cvx_prob} for $t = t_{k+1}$ within the same accuracy by performing only \textbf{one} inexact proximal-Newton (PN) step. 
Moreover, our scheme adaptively updates the regularization parameter to trace the path of solutions, towards the optimal solution of \eqref{eq:constr_cvx_prob}.
%\begin{align}\label{eq:barrier_cvx_subprob2}
%\xb^{k\!+\!1}_{t_{k\!+\!1}} \!:\approx\!   \mathrm{arg}\!\!\!\!\min_{\!\!\!\!\!\!\!\xb \in \dom{F}}{\!} \Big\{ &\nabla f(\xb^{k}_{t_k}\!)^T\!(\xb \!-\! \xb^{k}_{t_k}\!)  \!+\! \frac{1}{2}(\xb \!-\! \xb^{k}_{t_k}\!)^T\!\nabla^2{f}(\xb^{k}_{t_k}\!)(\xb \!-\! \xb^{k}_{t_k}\!)
%+ \frac{1}{t_{k\!+\!1}}g(\xb)\Big\},
%\end{align}
%\begin{small}
%\begin{align}\label{eq:barrier_cvx_subprob2}
%\xb^{k + 1}_{t_{k + 1}}  :\approx    \mathrm{arg}    \min_{       \xb \in \dom{F}}{ } \Big\{ &\nabla f(\xb^{k}_{t_k} )^T (\xb  -  \xb^{k}_{t_k} )   +  \frac{1}{2}(\xb  -  \xb^{k}_{t_k} )^T \nabla^2{f}(\xb^{k}_{t_k} )(\xb  -  \xb^{k}_{t_k} )
%+ \frac{1}{t_{k + 1}}g(\xb)\Big\},
%\end{align}
%\end{small}
%where $\dom{F} := \dom{f}\cap\dom{g}$, and  "$:\approx$" means that $\xb^{k+1}_{t_{k+1}}$ is an approximation to the true solution $\bar{\xb}^{k+1}_{t_{k+1}}$ of the minimization problem (c.f. Definition \ref{de:inexact_sol} below). 
%%This is the workhorse of our framework in a manner similar to the Newton schemes for standard path-following interior point methods \cite{Nesterov2004,Nesterov1994}. 
%I.e., we only need to solve \eqref{eq:barrier_cvx_subprob2} up to the required accuracy to obtain the new point $\xb^{k+1}_{t_{k+1}}$.

But, \emph{how such a scheme is advantageous in practice?} 
% We now highlight the two salient features of our scheme that set us apart from existing approaches:
First, due to the non-smoothness of the objective function, solving \eqref{eq:cvx_prob} can be a strenuous task. However, using proximal-Newton/gradient strategies to solve approximations of \eqref{eq:cvx_prob} has been a major research area over the last decade, broadly known as composite optimization, where many accurate and scalable algorithms are customized for different $g$ functions \cite{Beck2009,Becker2011b,Nesterov2007}. These methods are theoretically as fast as the advanced ``Hessian-free'' IPM techniques, which use conjugate gradients, since $\nabla^2f(\xb)\succ 0$ for self concordant-barriers. Our path following scheme leverages such algorithms as a black-box to solve \eqref{eq:constr_cvx_prob}: as we handle the non-smooth term $g$ directly with proximity operators, we retain the original problem structure (i.e., we do not inflate problem dimensions or add additional constraints by lifting the nonsmooth term via slack variables).

Second, adaptively updating the regularization parameter in composite optimization problems has itself attracted a great deal of interest; cf., the class of  homotopy and continuation methods \cite{Hale2008}. Many of these approaches lose their theoretical guarantees (if any) when the composite minimization problem has a self-concordant smooth term instead of a Lipschitz continuous gradient smooth term.  Our scheme provides a rigorous way of updating the regularizer weights and can be easily adapted for applications with self-concordant data terms  \cite{Harmany2012,Ravikumar2011,Stadler2012}, where none of these methods apply. 

%%% Contribution.
\vskip.1in
\noindent\textbf{Our contributions:}  
Our specific contributions in this paper are as follows:
\begin{itemize}
\item[(a)] We extend the notion of path-following scheme to handle composite forms in order to approximately track the solution trajectory of \eqref{eq:cvx_prob}.
As a consequence, we obtain an approximate solution of \eqref{eq:constr_cvx_prob} by controlling the parameter $t$ to $0^{+}$.

\item[(b)] We provide an explicit formula to adaptively update the parameter $t$ with convergence guarantees, without any manual tuning strategy. 
 
\item[(c)] We provide a theoretical analysis of the  worst-case analytical  complexity of our scheme to obtain a sequence of approximate solutions, as $t$ varies, while allows one to inexactly compute the proximal-Newton directions up to a given accuracy.
The worst-case analytical complexity of our method remains the same order as in conventional path-following interior point methods \cite{Nesterov2004}.
\end{itemize}

%% Outline.
\vskip0.15cm
\noindent\textbf{Paper outline. } 
Section \ref{sec:background} recalls the definitions of self-concordant functions and  barriers and sets up optimization preliminaries.  
Section \ref{sec:inexact_full_step_prox_Newton} deals with the inexact proximal-Newton iteration scheme for solving \eqref{eq:cvx_prob} at a fixed value of the parameter $t$.
Section \ref{sec:phase2_analysis} presents the path-following framework  with inexact proximal-Newton iterations and analyzes its convergence and worst-case analytical complexity.
Section \ref{sec:constrained_case} specifies our framework to solve constrained convex minimization problems of the form \eqref{eq:constr_cvx_prob}. 
Section \ref{sec:application} presents  numerical experiments that highlight the strengths and weaknesses of our framework.
Technical proofs are given in the appendix.
%%%%%%%%%%%%%%%%%%%%%%%%%%%%%%%%%%%%%%%%%%%%%%%%%%%%%%%%%%%%%
%%% 2. Background.
%%%%%%%%%%%%%%%%%%%%%%%%%%%%%%%%%%%%%%%%%%%%%%%%%%%%%%%%%%%%%
\section{Preliminaries}\label{sec:background} 
In this section, we set up the necessary notation, definitions and basic properties related to problem \eqref{eq:cvx_prob}.

%%%%%%%%%%%%%%%%%%%%%%%%%%%%%%%%%%%%%%%%%%%%%%%%%%%%%%%%%%%%%
%%% 2.1. Basic definitions.
%%%%%%%%%%%%%%%%%%%%%%%%%%%%%%%%%%%%%%%%%%%%%%%%%%%%%%%%%%%%%
\subsection{Basic definitions} 
Given $\mathbf{x}, \mathbf{y} \in \mathbb{R}^n$, we use $\mathbf{x}^T\mathbf{y}$ to denote the inner product in $\mathbb{R}^n$. 
For a proper, lower semicontinuous convex function $f$, we denote its domain by $\dom{f}$ (i.e., $\dom{f} := \set{\xb\in\mathbb{R}^n ~|~ f(\xb) < + \infty}$ and its subdifferential at $\xb$ by $\partial{f}(\xb) := \set{\vb\in\mathbb{R}^n~|~ f(\yb) \geq f(\xb) + \vb^T(\yb - \xb), ~\forall\yb\in\dom{f} }$. We also define $\mathrm{Dom}(f) := \mathrm{cl}(\dom{f})$ the closure of $\dom{f}$ \cite{Rockafellar1970}.

For a given twice differentiable function $f$ such that $\nabla^2f(\xb) \succ 0 $ at $\xb\in\dom{f}$, we define the
local norm $\norm{\ub}_{\xb} := [\ub^T\nabla^2 f(\xb)\ub]^{1/2}$ for any $\ub\in\mathbb{R}^n$ while the dual norm is given by
$\norm{\vb}_{\xb}^{*} := \max_{\norm{\ub}_{\xb} \leq 1}\ub^T\vb = [\vb^T\nabla^2f(\xb)^{-1}\vb]^{1/2}$ for any $\vb\in\mathbb{R}^n$. 
It is clear that the Cauchy-Schwarz inequality holds, i.e., $\ub^T\vb \leq \norm{\ub}_{\xb}\norm{\vb}^{*}_{\xb}$.
For our analysis, we also use two simple convex functions $\omega(t) := t - \ln(1+t)$ for $t \geq 0$ and $\omega_{*}(t) := -t - \ln(1 - t)$ for $t \in [0, 1)$, which are strictly increasing in their domain. 

An important concept in this paper is the self-concordance property \cite{Nesterov2004,Nesterov1994}.
% Definition 2.1.
\begin{definition}\label{de:concordant}
A convex function $\varphi: \dom{\varphi}\subseteq\R \rightarrow \mathbb{R} $ is called standard self-concordant if $\abs{\varphi'''(\tau)} \leq 2\varphi''(\tau)^{3/2}$  for all $\tau\in\dom{\varphi}$. 
A function $f:\dom{f}\subseteq \mathbb{R}^{n} \rightarrow \mathbb{R}$ is self-concordant if $\mathbf{x}\in\dom{f}$ and $\mathbf{v}\in\mathbb{R}^n$ such that $\xb + \tau\vb \in \dom{f}$, the function $\varphi(\tau) := f(\xb + \tau \vb)$ is standard self-concordant .
\end{definition}

% Definition 2.2.
\begin{definition}
A standard self-concordant function $f : \dom{f}\to\R$ is a $\nu$-self-concordant barrier for the set $\mathrm{Dom}(f)$ with parameter $\nu > 0$, if 
\begin{equation*}
\sup_{\ub \in\mathbb{R}^n} \left \{2\nabla{f}(\xb)^T\ub - \Vert \ub\Vert_{\xb}^2\right \} \leq \nu, ~~\forall \xb\in\dom{f}.
\end{equation*} 
\end{definition}
%% End of the definition.

We note that when $\nabla^2 f$ is non-degenerate (particularly, $\dom{f}$ contains no straight line \cite[Theorem 4.1.3.]{Nesterov2004}), a $\nu$-self-concordant function $f$ satisfies
\begin{align}{\label{eq:used}}
\norm{\nabla f(\xb)}_{\xb}^{\ast} \leq \sqrt{\nu}, ~~\forall \xb\in\dom{f}.
\end{align}
For more details on self-concordant functions and self-concordant barriers, we refer the reader to Chapter 4 of \cite{Nesterov2004}. 
Several simple sets are equipped with a self-concordant barrier. 
For instance, $f_{\mathbb{R}^n_{+}}(\xb) := -\sum_{i=1}^n\log(x_i)$ is an $n$-self-concordant barrier of the orthogonal cone $\mathbb{R}^n_{+}$, $f(\xb, t) = -\log(t^2 - \norm{\xb}_2^2)$ is a $2$-self-concordant barrier of the Lorentz cone $\mathcal{L}_{n+1} := \set{(\xb, t) \in\mathbb{R}^n\times\mathbb{R}_{+} ~|~ \norm{\xb}_2 \leq t}$, and the semidefinite cone $\mathcal{S}^n_{+}$ is endowed with an $n$-self-concordant barrier $f_{\mathcal{S}_{+}^n}(\mathbf{X}) := -\log\det(\mathbf{X})$. 

Given these definitions, we are now ready to state our main assumption used throughout this paper.
%% Assumption 3.1.
\begin{assumption}\label{as:A2}
The function $f$ in \eqref{eq:cvx_prob} is a $\nu$-self-concordant barrier with $\nu > 0$ for $\mathrm{Dom}(f)$. 
The function $g : \mathbb{R}^n\to\mathbb{R}\cup\set{+\infty}$ is proper, lower semi-continuous and convex.
\end{assumption}

% 2.1. Optimality condition of 1.1
\subsection{Optimality condition of \eqref{eq:cvx_prob}} 
Given $t > 0$, we assume that problem \eqref{eq:cvx_prob} has a solution $\xb^{*}_t$. Since $f$ is strictly convex, this solution is also unique.
The optimality condition of \eqref{eq:cvx_prob} can be written as
\begin{equation}\label{eq:optimality_for_ln_barrier}
\mathbf{0} \in \nabla{f}(\xb^{*}_t) + t^{-1}\partial{g}(\xb^{*}_t).
\end{equation} 
The formula \eqref{eq:optimality_for_ln_barrier} expresses a \textit{monotone inclusion} \cite{Facchinei2003}.
If $g$ is smooth, \eqref{eq:optimality_for_ln_barrier} reduces to $\nabla{f}(\xb^{*}_t) + t^{-1}\nabla{g}(\xb^{*}_t) = \mathbf{0}$, a system of nonlinear equations. 
Any $\xb^{*}_t$ satisfying \eqref{eq:optimality_for_ln_barrier} is called a stationary point of \eqref{eq:cvx_prob}, which is also a global optimum of \eqref{eq:cvx_prob}, for given $t > 0$. 
Let $\dom{F} \equiv \dom{F(\cdot; t)} := \dom{f}\cap\dom{g}$ for fixed $t > 0$. Then $\xb^{*}_t \in\dom{F}$.

% Definition 2.3.
\begin{definition}\label{de:prox_oper}
Let $\xb \in \dom{f}$ such that $\nabla^2{f}(\xb) \succ 0$ and let $\mathbf{s} \in \mathbb{R}^n$ be an arbitrary given point. 
We define the operator $P^g_{\xb}(\cdot; t)$ with an input $\mathbf{s}$ and a parameter $t > 0$ as follows:
\begin{equation}\label{eq:P_x}
P^g_{\xb}(\mathbf{s}; t) = \mathrm{arg}\!\!\min_{\yb \in \mathbb{R}^n}\set{ t^{-1}g(\yb) + \frac{1}{2}\yb^T\nabla^2f(\xb)\yb - \mathbf{s}^T\yb}.
\end{equation} 
\end{definition}
Since $\nabla^2{f}(\xb) \succ 0$, we can write \eqref{de:prox_oper} as
\begin{equation*}
P^g_{\xb}(\mathbf{s}; t) = \mathrm{arg}\!\!\min_{\yb \in \mathbb{R}^n}\set{ g(\yb) + \frac{t}{2}\norm{\yb - \nabla^2f(\xb)^{-1}\mathbf{s}}_{\xb}^2},
\end{equation*} 
which requires to compute the proximal operator of $g$ at $\nabla^2f(\xb)^{-1}\mathbf{s}$ w.r.t. the weighted norm $\norm{\cdot}_{\xb}$.
Given $\xb$ and $\mathbf{s}$ as defined above, we define the following mapping:
\begin{equation}\label{eq:S_x}
S_{\xb}(\mathbf{s}) := \nabla^2{f}(\xb)\mathbf{s} - \nabla{f}(\mathbf{s}). 
\end{equation}
The optimality condition in \eqref{eq:optimality_for_ln_barrier} implies the following fixed-point characterization of the mapping $P^g_{\xb}(\cdot; t)$. The proof can be found in \cite{TranDinh2013c}.

% Lemma 2.4.
\begin{lemma}
Let $t > 0$ be fixed. Then, the mapping $P_{\xb}^g(\cdot;t)$ defined in \eqref{eq:P_x} is co-coercive and therefore nonexpansive w.r.t. the local norms, i.e.:
\begin{align}
\mathrm{[co\textrm{-}coercive]:} ~~~~~& \left(P^g_{\xb}(\ub; t) - P^g_{\xb}(\vb; t)\right)^T(\ub - \vb) \geq \norm{P^g_{\xb}(\ub; t) - P^g_{\xb}(\vb; t)}_{\xb}^2, \label{eq:cocoercive}\\
\mathrm{[nonexpansive]:}~~ &\norm{P^g_{\xb}(\ub; t) - P^g_{\xb}(\vb; t)}_{\xb} \leq \norm{\ub - \vb}_{\xb}^{*}, ~\forall \ub, \vb \in \mathbb{R}^n. \label{eq:nonexapansiveness}
\end{align}
Furthermore, the following fixed-point characterization holds:
\begin{equation}\label{eq:fixed_point_xstar}
\xb^{*}_t = P^g_{\xb^{*}_t}\left( S_{\xb^{*}_t}(\xb^{*}_t); t\right),
\end{equation}
where $\xb^{\ast}_t \in \dom{F}$ is the minimizer of \eqref{eq:cvx_prob}, i.e., $\xb^{\ast}_t \in \mathrm{arg}\!\!\displaystyle\min_{\xb \in \mathbb{R}^n} F(\xb; t)$. 
\end{lemma} 

\noindent For our convergence analysis, we also need the following result.
% Lemma 2.5.
\begin{lemma}\label{le:lower_bound}
For fixed $t > 0$, let $\xb^{*}_t$ be the unique solution of \eqref{eq:cvx_prob}. Then, for any $\xb\in\dom{F}$, the following estimate holds:
\begin{equation}\label{eq:lower_bound}
\omega\left(\norm{\xb - \xb^{*}_t}_{\xb^{*}_t}\right) \leq F(\xb; t) - F(\xb^{*}_t; t).
\end{equation}
\end{lemma}

% Proof of Lemma 2.5.
\begin{proof}
By the self-concordance property of $f$ for any $\xb\in\dom{F}$, the convexity of $g$ and \eqref{eq:optimality_for_ln_barrier} it follows that there exists $\vb_t^{*} \in \partial{g}(\xb^{*}_t)$ such that
\begin{align*}
F(\xb;t) - F(\xb^{*}_t;t) &= f(\xb) - f(\xb^{*}_t) + t^{-1}\left(g(\xb) - g(\xb^{*}_t) \right) \nonumber\\
&\geq \left(\nabla{f}(\xb^{*}_t) + t^{-1}\vb_t^{*}\right)^T(\xb - \xb^{*}_t) + \omega\left(\norm{\xb - \xb^{*}_t}_{\xb^{*}_t}\right), ~~\vb_t^{*} \in \partial{g}(\xb^{*}_t), \nonumber\\
&\overset{\tiny\eqref{eq:optimality_for_ln_barrier}}{=} \omega\left(\norm{\xb - \xb^{*}_t}_{\xb^{*}_t}\right),
\end{align*}
which is indeed \eqref{eq:lower_bound}.
\end{proof}
% End of the proof.

%%%%%%%%%%%%%%%%%%%%%%%%%%%%%%%%%%%%%%%%%%%%%%%%%%%%%%%%%%%%%
%%% 2. Inexact full-step proximal-Newton iterations}
%%%%%%%%%%%%%%%%%%%%%%%%%%%%%%%%%%%%%%%%%%%%%%%%%%%%%%%%%%%%%
\section{Proximal-Newton iterations for fixed $t$}\label{sec:inexact_full_step_prox_Newton}
Let us consider the unconstrained problem \eqref{eq:cvx_prob} for a given fixed parameter value $t > 0$. 
Since $f$ is self-concordant, we can approximate it around $\xb_t \in \dom{F}$ via the second order Taylor series expansion:
\begin{align}\label{eq:quad_approx}
Q(\xb; \xb_t) := f(\xb_t) \!+\! \nabla f(\xb_t)^T (\xb \!-\! \xb_t) \!+\!  \frac{1}{2}(\xb \!-\! \xb_t)^T \nabla^2 f(\xb_t)(\xb \!-\! \xb_t).
\end{align}
Given this quadratic surrogate of $f$, we can approximate $F(\xb; t)$ around $\xb_t$ as:
\begin{align}\label{eq:F_x}
\widehat{F}(\xb; t, \xb_t) := Q(\xb; \xb_t) + t^{-1}g(\xb).
\end{align} 
Starting from an arbitrary initial point $\xb^0_t\in\dom{F}$ and given a fixed value $t > 0$, the \textit{inexact} full-step proximal-Newton method for solving \eqref{eq:cvx_prob}  generates a sequence of points, by approximately minimizing the composite quadratic model \eqref{eq:F_x} as
\begin{align}\label{eq:barrier_cvx_subprob1}
\left\{
	\begin{array}{ll}
		\xb^+_t  \approx \argmin_{\xb\in\dom{F}}  \widehat{F}(\xb; t, \xb_t). \\
		\xb_t \leftarrow \xb^+_t
	\end{array}
\right.
\end{align} 
Here, the ``approximation'' sense ($\approx$) highlights the inability of numerical methods to iteratively solve \eqref{eq:barrier_cvx_subprob1} with \emph{exact accuracy} and $\leftarrow$ indicates the assignment operator. 

Assume that $\nabla^2f(\xb_t) \succ 0$, then the minimization problem in \eqref{eq:barrier_cvx_subprob1}  is a strongly convex program and it has the  unique \textit{exact} solution $\bar{\xb}^+_t$, i.e., 
\begin{align*}
\left\{
	\begin{array}{ll}
		\bar{\xb}^+_t  := \argmin_{\xb\in\dom{F}}  \widehat{F}(\xb; t, \xb_t). \\
		\xb_t \leftarrow \bar{\xb}^+_t
	\end{array}
\right.
\end{align*} Moreover,  the following optimality condition holds
\begin{equation}\label{eq:opt_cond_subprob}
\mathbf{0} \in \nabla{f}(\xb_t) + \nabla^2f(\xb_t)(\bar{\xb}^+_t - \xb_t) + t^{-1}\partial{g}(\bar{\xb}^+_t).
\end{equation}
Due to \eqref{eq:optimality_for_ln_barrier}, it is obvious to show that, if $\bar{\xb}^+_t \equiv \xb_t$, then  $\xb_t$ is the optimal solution of \eqref{eq:cvx_prob} for fixed $t$.

\begin{remark}
In Phase I of Algorithm \ref{alg:PF_PN_alg} below, we do not require $\nabla^2f(\xb)$ to be positive definite for all $\xb\in\dom{F}$ as long as the subproblem \eqref{eq:barrier_cvx_subprob1} has solution. 
The positive definiteness of $\nabla^2f(\xb)$ is only needed in Phase II of this algorithm, where the path-following iterations take place.
\end{remark}

The following table disambiguates the various notions introduced in this section:
\begin{table*}[!htb]
\vspace{-0.4cm}
%\caption{Notational convention\label{table:x_def}} 
\centering
\begin{tabular}{c c c} \toprule
\multicolumn{1}{c}{Notation} & \phantom{ab} & \multicolumn{1}{c}{Description} \\
\cmidrule{1-1} \cmidrule{3-3} 
\multicolumn{1}{c}{$\xb_t^{*}$} & & \multicolumn{1}{c}{Exact solution of \eqref{eq:cvx_prob} for fixed $t$.} \\ 
\multicolumn{1}{c}{$\bar{\xb}_t^+$} & & \multicolumn{1}{c}{Exact solution of \eqref{eq:F_x} around $\xb_t$ for fixed $t$.} \\ 
\multicolumn{1}{c}{$\xb_t^+$} & & \multicolumn{1}{c}{Inexact solution of \eqref{eq:F_x} around $\xb_t$ for fixed $t$.} \\ 
\bottomrule
\end{tabular}
\vskip-0.3cm
\end{table*}

% 3.1. Approximate solution characterizations for 3.1.
\subsection{Inexact solutions of \eqref{eq:barrier_cvx_subprob1}}
In practice, computing $\bar{\xb}_t^+$ is \textit{infeasible} except for special cases. 
Thus, we can only solve \eqref{eq:barrier_cvx_subprob1} up to a given accuracy $\delta \geq 0$
using algorithmic solutions such as fast proximal-gradient methods and alternating direction methods of multipliers \cite{Beck2009,Nesterov2004,Nesterov2007,Yang2011} in the following sense.

% Definition 3.1.
\begin{definition}\label{de:inexact_sol}
Given $t > 0$ and  a tolerance $\delta \geq 0$, a point $\xb^+_t\in\dom{F}$ is called a $\delta$-solution to \eqref{eq:barrier_cvx_subprob1} if 
\begin{equation}\label{eq:approx_sol}
\widehat{F}(\xb^+_t;t, \xb_t) - \widehat{F}(\bar{\xb}^+_t;t, \xb_t)  \leq \frac{\delta^2}{2}.
\end{equation}
\end{definition} 
% End of Definition 3.1.

A useful inequality for our subsequent developments is given in the next lemma.
% Lemma 3.2.
\begin{lemma}
Given fixed $t > 0$, the following inequality holds $\forall \xb \in \dom{F}$:
\begin{equation}\label{eq:inexact_sol2}
\frac{1}{2}\norm{\xb - \bar{\xb}^+_t}_{\xb_t}^2 \leq \widehat{F}(\xb; t, \xb_t) - \widehat{F}(\bar{\xb}^+_t; t, \xb_t), ~~\forall \xb \in \dom{F}.
\end{equation}
\end{lemma}

% The proof of Lemma 3.2.
\begin{proof}
Since $\nabla^2f(\xb_t) \succ 0$, the proof follows similar motions with the proof of Lemma \ref{le:lower_bound}, based on the optimality condition \eqref{eq:opt_cond_subprob} and the convexity of $g$.
\end{proof} 
% End of the proof.

This lemma, in combination with Definition \ref{de:inexact_sol}, indicates that, if we can find a $\delta$-solution $\xb^+_t$ using \eqref{eq:barrier_cvx_subprob1}, then
\begin{equation}\label{eq:computable_criterion2}
\norm{\xb^+_t - \bar{\xb}^+_t}_{\xb_t} \leq \delta.
\end{equation}

% 3.2. Contraction property of approximate proximal-Newton iterations.
\subsection{Contraction property of inexact proximal-Newton iterations}{\label{sec:subsub}}
In this subsection, we provide a theoretical characterization of the per-iteration behavior of the inexact full-step proximal-Newton scheme \eqref{eq:barrier_cvx_subprob1} for fixed $t > 0$.
Let $\xb_t^+ \in \dom{F}$ be a $\delta$-solution of \eqref{eq:barrier_cvx_subprob1} and let $\xb^{*}_t$ be the exact solution of \eqref{eq:cvx_prob}.
We define
\begin{equation}\label{eq:prox_Newton_decrement}
\lambda_t := \norm{\xb_t - \xb^{*}_{t}}_{\xb^{*}_{t}} \quad \text{and} \quad \lambda_t^+ := \norm{\xb_t^+ - \xb^{*}_{t}}_{\xb^{*}_{t}}, 
\end{equation} 
as the {\it weighted} distance between $\xb_t/\xb^{*}_t$ and $\xb_t^+/\xb^{*}_t$, respectively.  The following theorem characterizes the contraction properties of $\lambda_t$; the proof can be found in the appendix.

%% Theorem 3.2.
\begin{theorem}\label{th:quad_converg_FPNM}
Given $\xb_t \in \dom{F}$, let $\xb^+_t$ be a $\delta$-solution of \eqref{eq:barrier_cvx_subprob1} for a given $\delta \geq 0$. Then, if $\lambda_t \in [0, 1 - \frac{\sqrt{2}}{2})$, we have
\begin{equation}\label{eq:FPNM_estimate}
\lambda^+_t  \leq \frac{\delta}{1-\lambda_t} + \left(\frac{3 - 2\lambda_t}{1 - 4\lambda_t + 2\lambda_t^2}\right)\lambda_t^2.
\end{equation}
Moreover, the right-hand side of \eqref{eq:FPNM_estimate} is nondecreasing w.r.t.\ $\lambda_t$ and $\delta \geq 0$.
\end{theorem}

To illustrate the contraction properties of $\lambda_t$, we assume that the accuracy $\delta$ can be chosen such that $\delta := \xi \lambda_t$ for a given $\xi \in (0, 1)$. Furthermore, let us define $\varphi(\lambda, \xi) := \frac{\xi}{1 - \lambda} + \frac{3\lambda - 2\lambda^2}{1 - 4\lambda + 2\lambda^2}$ on $[0, 1 - \frac{\sqrt{2}}{2})$. Then, \eqref{eq:FPNM_estimate} can be rewritten as
\begin{equation}{\label{eq:contraction}}
\lambda_t^+ \leq \varphi(\lambda_t, \xi)\lambda_t.
\end{equation}
From \eqref{eq:contraction}, we observe that, if $\varphi(\lambda,\xi) \leq \omega < 1$ for $\lambda\in [0, 1-\sqrt{2}/2]$ and $\xi\in (0,1)$, the sequence of distances $\set{\Vert\xb^+_t - \xb^{*}_t\Vert_{\xb^{*}_t}}_{+}$ by solving \eqref{eq:barrier_cvx_subprob1} becomes contractive, i.e., it ensures the convergence of the proximal-Newton scheme \eqref{eq:barrier_cvx_subprob1}. To this end, we need to find a range of $\lambda_t$ values, (say $\Lambda$), such that $\varphi < 1$. Varying $\xi$, we can choose this range $\Lambda$: Since $\varphi$ is non-decreasing, the larger $\xi$ is, the smaller the range of $\Lambda$ becomes. This observation is illustrated in Figure \ref{fig:phi_func}. 
\begin{figure}[!ht]
\begin{center}
\centerline{\includegraphics[width = 8cm, height = 5cm]{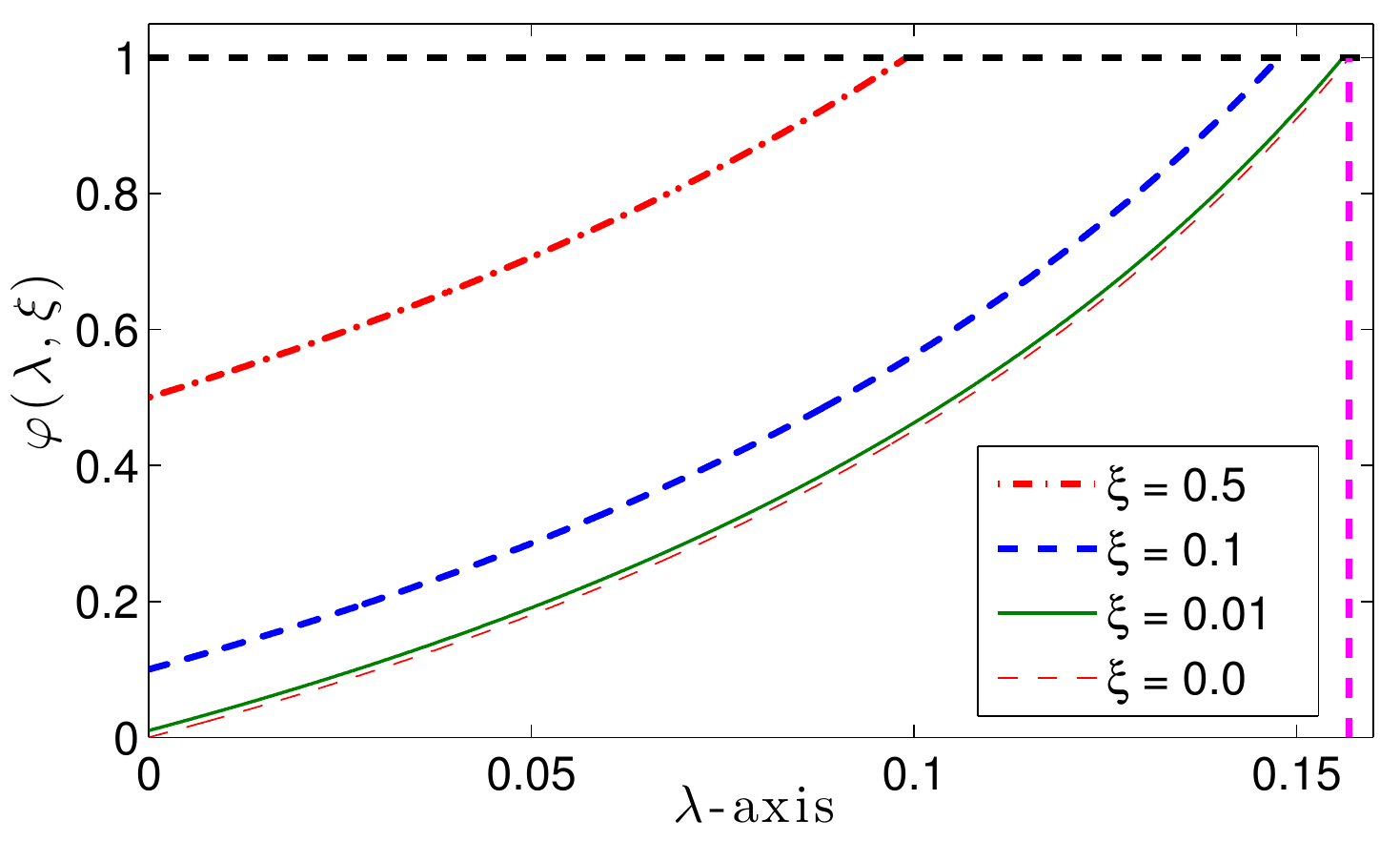}}
\vskip-0.4cm
\caption{The behavior of the contraction factor function $\varphi(\lambda, \xi)$ at different values of $\xi$.}\label{fig:phi_func}
\end{center}
\vskip -0.1in
\end{figure} 
For $\xi \in [0, 0.5]$, the interval where $\varphi < 1$ varies from $[0, 0.098]$ to $[0, 0.154]$. We observed also that, for $\xi = 0.001$, the value of $\varphi$ is very close to the case of $\xi = 0$. In practice, this suggests that if we set the accuracy $\delta < 10^{-3}$, the inexact scheme performs closely to the ideal case (i.e., $\delta = 0$).

Theoretically, if we assume that the subproblem \eqref{eq:barrier_cvx_subprob1} is solved exactly, then the estimate \eqref{eq:FPNM_estimate} reduces to $\bar{\lambda}_t^+ \leq \left(\frac{3-2\bar{\lambda}_t}{1 - 4\bar{\lambda}_t + 2\bar{\lambda}_t^2}\right)\bar{\lambda}_t^2$, where $\bar{\lambda}_t := \norm{\bar{\xb}_t - \xb^{*}_{t}}_{\xb^{*}_{t}}$ and $\bar{\lambda}_t^+ := \norm{\bar{\xb}_t^+ - \xb^{*}_{t}}_{\xb^{*}_{t}}$. The algorithms and the convergence theory corresponding to this case were studied in \cite{TranDinh2013c}.

Now, within this section, we define $\xb^j_t := \xb_t$, $\xb^{j+1}_t := \xb^{+}_t$, $\lambda_t^j :=\lambda_t$, $\lambda_t^{j+1} := \lambda_t^{+}$ and $\delta_j := \delta$, where the index $j$ indicates the iteration counter for fixed $t$ starting from $j := 0$.
An important consequence of Theorem \ref{th:quad_converg_FPNM} is the following corollary.

% Corollary
\begin{corollary}\label{co:quad_convergence}
For a fixed $t > 0$ and a given constant $c > 0$, let $\{\xb_t^j\}_{j\geq 0}$ be a sequence of $\delta_j$-solutions, generated by solving \eqref{eq:barrier_cvx_subprob1} approximately.
\begin{itemize}
\item[$(a)$] If we choose $\delta_j$ and $\xb^0_t$ such that $\delta_j \leq 0.15\lambda^j_t$ for $j\geq 0$ and $\lambda_t^0 \leq 0.1427$, then the sequence of approximate solutions $\{\xb^j_t\}_{j\geq 0}$ in \eqref{eq:barrier_cvx_subprob1} converges to $\xb^{*}_t$ at a linear rate.

\item[$(b)$] If we choose $\delta_j$ and $\xb^0_t$ such that $\delta_j \leq c \cdot (\lambda_t^j)^2$ for $j\geq 0$ and $$\lambda_t^0 \in \left[0, \min\set{0.15,  (1.177c + 6.068)^{-1}}\right],$$ then the sequence of approximate  solutions  $\{\xb^j_t\}_{j\geq 0}$  in \eqref{eq:barrier_cvx_subprob1} converges to $\xb^{*}_t$ at a quadratic rate.
\end{itemize}
\end{corollary}

% Proof of the corollary.
\begin{proof}
$(a)$.~For $\delta_j := 0.15\lambda^j_t$, we observe that $\varphi(\lambda, 0.15)$ is increasing for $\lambda \in \Lambda := [0, 0.1427]$ and $\varphi(\lambda, 0.15) < 1$ for all $\lambda \in \Lambda$. Therefore, it follows from \eqref{eq:FPNM_estimate} that $\lambda_t^{j+1} \leq \max_{\lambda \in \Lambda}\left\{ \varphi(\lambda, 0.15)\right\}\cdot\lambda_t^j$ for $j\geq 0$, which implies the sequence $\set{\lambda_t^j}$ converges to zero at a linear rate. Since $\lambda_t^j := \Vert\xb^j_t - \xb^{*}_t\Vert_{\xb^{*}_t}$, the sequence $\{\xb^j_t\}_{j\geq 0}$ also converges to $\xb^{*}_t$ at a linear rate (in the weighted norm).

$(b)$. ~For $\lambda_t^j \in [0, 0.15]$, we can see that the weight factor of the second term in the right hand side of \eqref{eq:FPNM_estimate}, i.e., $\frac{3 - 2\lambda_t^j}{1 - 4\lambda^j_t + 2(\lambda_t^j)^2}$, is increasing and moreover, $\frac{3 - 2\lambda_t^j}{1 - 4\lambda_t^j + 2(\lambda^j_t)^2} \leq 6.068$.  Thus, for $\delta_j \leq c \cdot (\lambda_t^j)^2$, we have
\begin{align*}
\lambda_t^{j+1} \leq \left( \frac{c}{1-\lambda_t^j} + 6.068\right)(\lambda_t^j)^2 \leq \left( 1.177c  + 6.068\right)(\lambda_t^j)^2.
\end{align*}
From this inequality, we can easily check that, if $\lambda_t^0 \leq \min\set{0.15, (1.177c + 6.068)^{-1}}$ then $\lambda_t^{j+1}\leq C \cdot (\lambda_t^j)^2$ for $C := 1.177c  + 6.068$. This estimate shows that  the sequence $\set{\lambda_t^j}_{j\geq 0}$ converges to zero at a quadratic rate. Consequently, the sequence $\{\xb^j_t\}_{j\geq 0}$ also converges to $\xb^{*}_t$ at a quadratic rate (in the weighted norm).
\end{proof}
% End of the proof.

%%%%%%%%%%%%%%%%%%%%%%%%%%%%%%%%%%%%%%%%%%%%%%%%%%%%%%%%%%%%%%%%%%%%%%%%%
%% 4. The proximal Newton iterations and path-following scheme.
%%%%%%%%%%%%%%%%%%%%%%%%%%%%%%%%%%%%%%%%%%%%%%%%%%%%%%%%%%%%%%%%%%%%%%%%%
\section{A proximal path following framework}\label{sec:phase2_analysis}
Our discussion so far focuses on the case of minimizing \eqref{eq:barrier_cvx_subprob1} for a fixed $t > 0$. Nevertheless, in order to solve the initial problem \eqref{eq:constr_cvx_prob}, one requires to trace the sequence of solutions as $t_k \downarrow 0^{+}$. For smooth self-concordant barrier function minimization problems, Nesterov and Nemirovskii in \cite{Nesterov2004,Nesterov1994}  presented a path following strategy where  a {\it single Newton step} per iteration is used, for each well-chosen penalty parameter $t_k$. Here, we adopt a similar strategy to handle composite self-concordant barrier problems of the form \eqref{eq:cvx_prob} with a possibly nonsmooth convex function $g$, \emph{mutatis mutandis}. 

Our contribution lies at the adaptive selection of $t$: given an anchor point $\xb_{t_k}$ obtained using $t_k > 0$, we compute an approximate solution $\xb_{t_{k+1}}$ as:
\begin{small}
\begin{align}\label{eq:barrier_cvx_subprob2}
\xb_{t_{k \!+\! 1}}  \!:\approx\!   \argmin_{\xb \in \dom{F}}\Big\{ &\nabla f(\xb_{t_k} )^T (\xb  \!-\!  \xb_{t_k} ) \!+\!  \frac{1}{2}(\xb \!-\! \xb_{t_k} )^T \nabla^2{f}(\xb_{t_k} )(\xb  \!-\!  \xb_{t_k} ) + t_{k +  1}^{-1}g(\xb)\Big\},
\end{align}
\end{small}
where $t_{k+1}$ is adaptively updated per iteration, based on $t_k$;  ``$:\approx$'' is defined as in Definition \ref{de:inexact_sol}; and the index $k$ is to distinguish it from the index $j$ for fixed $t$. 
In stark contrast, classical path-following (homotopy or continuation) methods \cite{Fiacco1983,Guddat1990} usually discretize the parameter $t$ \emph{a priori} and then solve \eqref{eq:cvx_prob} over this grid.

Our proximal path following scheme goes through the following motions: Starting from an initial value $ t  = t_0$ and an anchor point $\xb^0$, we solve \eqref{eq:barrier_cvx_subprob2} to obtain an approximate solution $\xb_{t_0}$ to $\bar{\xb}_{t_0}$, i.e., the exact solution of \eqref{eq:barrier_cvx_subprob1} at $t = t_0$. The initial selection of $t_0$ is generally problem dependent. Section \ref{subsec:initial_point} describes a procedure on how to compute a good starting point $\xb_{t_0}$. Then, at the $k$-th iteration and given $\xb_{t_k}$, the scheme uses $t_k$ to compute $t_{k+1}$ parameter and performs a \textit{single} proximal-Newton (PN) iteration (i.e., solving \eqref{eq:barrier_cvx_subprob1} once for $t = t_{k+1}$) to approximately compute $\xb_{t_{k+1}}$. Moreover, for each iteration, we provably show the proximity of $\xb_{t_k}$ to $\xb^{*}_{t_k}$. 
This strategy is illustrated in Figure \ref{fig:path_following}.

%+ Picture of the proof.
\begin{figure}[ht]
\vskip-0.3cm
\hskip-1.5cm
\setlength{\unitlength}{1.1mm}
\begin{picture}(70, 40)
  \linethickness{0.075mm}
  \put(20, 0){\vector(0, 1){36}}
  \put(21,35){\color{blue}$\xb^{*}_t$}
  \put(20, 0){\vector(1, 0){50}} 
  \put(68, 2){$t$}
  {\color{blue}
  \linethickness{0.3mm}
  \qbezier(25, 5)(35, 27)(70, 28)
  }
  \linethickness{0.075mm}
  \put(25,12.2){\line(2,3){5}}
  \multiput(25,0)(0,1){13}{\circle*{0.01}}
  \multiput(30,0)(0,1){20}{\circle*{0.01}}
  \put(45,27.5){\line(6,1){15}}
  \multiput(46,0)(0,1){28}{\circle*{0.01}}
  \multiput(61,0)(0,1){31}{\circle*{0.01}}
  \put(25,12){\circle*{1}}
  {\color{red}\put(25,5){\circle*{1}}}
  \put(30,20){\circle*{1}}
  {\color{red}\put(30,13){\circle*{1}}}
  {\color{red}\put(45,23.7){\circle*{1}}}
  \put(45,27.5){\circle*{1}}
  \put(60,30){\circle*{1}}
  {\color{red}\put(60,27.5){\circle*{1}}}
  \multiput(30,20)(1,0.5){16}{\circle*{0.1}} 
  \put(24,-3){{\scriptsize$t_0$}}
  \put(29,-3){{\scriptsize$t_1$}}
  \put(44,-3){{\scriptsize$t_k$}}
  \put(59,-3){{\scriptsize$t_{k+1}$}}
  \put(20.5, 14){$\xb_{t_0}$}
  \put(19.5, 2){$\color{blue}\xb^{*}_{t_0}$}
  \put(44, 29.5){$\xb_{t_k}$}
  \put(45.3, 20.5){\color{blue}$\xb^{*}_{t_k}$}
  \put(59, 31.5){$\xb_{t_{k+1}}$}
  \put(60.5, 24.3){$\color{blue}\xb^{*}_{t_{k+1}}$}
  \put(18.5,-3){\scriptsize$0$}
  \put(42.0,25.5){\color{red}\makebox(0,0){{\tiny[1]}$\left\{\right.$}}
  \put(62.3,28.8){\color{red}\makebox(0,0){$\left\}\right.${\tiny[2]}}}
  \put(45,0){\color{red}\tiny$\overbrace{\rule{16.3mm}{0cm}}$}
  \put(51.5,2.2){\color{red}\tiny[4]}
  \put(58,25.3){\color{red}\makebox(0,0){{\tiny[3]}$\left\{\right.$}}
  {\color{red}\put(59.7, 23){\circle*{1}}}
  % These are the lables.
  \put(76,22){\scriptsize{\color{red}[1]}~:~$\lambda_{t_k}^+ = \Vert \xb_{t_k}-{\color{blue}\xb^{*}_{t_k}}\Vert_{\xb^{*}_{t_k}}$}  
  \put(76,18){\scriptsize{\color{red}[2]}~:~$\lambda_{t_{k+1}}^+ = \Vert \xb_{t_{k+1}}-{\color{blue}\xb^{*}_{t_{k+1}}}\Vert_{\xb^{*}_{t_{k+1}}}$}
  \put(76,14){\scriptsize{\color{red}[3]}~:~$\color{blue}\Delta_k = \Vert \xb^{*}_{t_{k+1}} - \xb^{*}_{t_{k}}\Vert_{\xb^{*}_{t_k}}$}
  \put(76,10){\scriptsize{\color{red}[4]}~:~$d_k := \abs{t_{k+1} - t_k}$}
  \put(74,6){\scriptsize{\color{red}}~~}\put(76,6){{\color{blue}\linethickness{0.3mm}\line(1,0){5}}~~\textrm{\scriptsize{The true solution trajectory  $\xb^{*}_{t_k}$}}}
  \put(74,2){\scriptsize{\color{red}}~~}\put(76,2){{\color{black}\linethickness{0.3mm}\line(1,0){5}}~~\textrm{\scriptsize{Approximate solution sequence $\{\xb_{t_k}\}$}}}
  % Put a text.
 \put(30,35){\color{black}\scriptsize{\textbf{one} inexact PN iteration}}
\put(45, 34){\color{red}\vector(0, -1){3}} 
\end{picture}
\vskip 0.2cm
\caption{The approximate sequence $\{x_{t_k}\}_{k\geq 0}$ along the solution trajectory $x^{*}_{t_k}$.}\label{fig:path_following}   
\end{figure}  

%%%%%%%%%%%%%%%%%%%%%%%%%%%%%%%%%%%%%%%%%%%%%%%%%%%%%%%%%%%%%%%%%%
%%% 3.3. The path-following scheme.
%%%%%%%%%%%%%%%%%%%%%%%%%%%%%%%%%%%%%%%%%%%%%%%%%%%%%%%%%%%%%%%%%%%%%%%
\subsection{Quadratic convergence region}
Based on \eqref{eq:prox_Newton_decrement}, we define $\lambda_{t_{k+1}} := \Vert \xb_{t_{k}} - \xb^{*}_{t_{k+1}}\Vert_{\xb^{*}_{t_{k+1}}}$ and $\lambda_{t_{k+1}}^{+} := \Vert\xb_{t_{k+1}} - \xb^{*}_{t_{k+1}}\Vert_{\xb^{*}_{t_{k+1}}}$. 
Given these definitions, for $t \equiv t_{k+1}$ and given the analysis in Section \ref{sec:subsub}, \eqref{eq:FPNM_estimate} becomes
\begin{equation}\label{eq:main_estimate1}
\lambda_{t_{k+1}}^{+} \leq \frac{\delta_k}{1 - \lambda_{t_{k+1}}} + \left(\frac{3 - 2\lambda_{t_{k+1}}}{1 - 4\lambda_{t_{k+1}} + 2\lambda_{t_{k+1}}^2}\right)\lambda_{t_{k+1}}^2,
\end{equation}
provided that $0 \leq \lambda_{t_{k+1}} < 1 - \sqrt{2}/2$. 
The index $k+1$ shows that we first update the parameter $t$ from $t_k$ to $t_{k+1}$ and then perform one PN step by solving \eqref{eq:barrier_cvx_subprob1} up to a given accuracy $\delta_k \geq 0$.

Let us define the weighted distance between two solutions $\xb^{*}_{t_{k+1}}$ and $\xb^{*}_{t_k}$ w.r.t. two different values $t_{k+1}$ and $t_k$ of $t$ as
\begin{equation}\label{eq:Delta}
\Delta_k := \Vert \xb^{*}_{t_{k+1}} - \xb^{*}_{t_k}\Vert_{\xb^{*}_{t_{k+1}}}.
\end{equation} 
The following theorem shows that, for a range of values for $\Delta_k$ and $\delta_k$, if $\lambda_{t_k}^+ \leq \beta$, then at the $(k+1)$-th iteration, we maintain the property $\lambda_{t_{k+1}}^+ \leq \beta$ for a given $\beta > 0$. The proof can be found in the appendix.

%%% Theorem 3.2.
\begin{theorem}\label{th:main_statement}
Let $\beta \in (0, 0.15]$ be fixed.
Assume that $\delta_k$ and $\Delta_k$ satisfy $\delta_k \leq 0.075\beta$ and $\Delta_k \leq \frac{\sqrt{\beta} - 2.581\beta}{2.581 + \sqrt{\beta}}$. Then, if $\lambda_{t_{k}}^+ \leq \beta$, then our scheme guarantees that $\lambda_{t_{k+1}} \leq \frac{1}{2.581}\sqrt{\beta}$ and  $\lambda_{t_{k+1}}^+ \leq\beta$. 
\end{theorem} 

Let us define $\mathcal{Q}^{t_k}_{\beta} := \set{\xb^k \in\dom{F} ~|~ \lambda_{t_k}^+ \leq \beta}$. We refer to $\mathcal{Q}^{t_k}_{\beta}$ as the \textit{quadratic convergence region} of the inexact proximal-Newton iterations \eqref{eq:barrier_cvx_subprob1} for solving \eqref{eq:cvx_prob}. 
For fixed $t_k > 0$, from Corollary \ref{co:quad_convergence}, we can see that if the starting point $\xb_{t_0}$ is chosen such that $\xb_{t_0}\in \mathcal{Q}^{t_k}_{\beta}$, then the whole sequence $\{\xb_{t_k}\}_{k\geq 0}$ generated belongs to $\mathcal{Q}^{t_k}_{\beta}$ and converges to $\xb^{*}_{t_k}$, the solution of \eqref{eq:cvx_prob}, at a quadratic rate.
In plain words, Theorem \ref{th:main_statement}  indicates that if the $\delta_k$-solution $\xb_{t_k}$ is in the quadratic convergence region $\mathcal{Q}^{t_k}_{\beta}$ at $\xb^{*}_{t_k}$ then, we can configure the proposed scheme such that the next $\delta_{k+1}$-solution $\xb_{t_{k+1}}$ remains in the quadratic convergence region $\mathcal{Q}^{t_{k+1}}_{\beta}$ at $\xb^{*}_{t_{k+1}}$.  

% 4.2. An adaptive update rule for t.
\subsection{An adaptive update rule for $t$}
Next, we show how we can  update the penalty parameter $t$ in our path-following scheme to ensure the condition on $\Delta_k$ in Theorem \ref{th:main_statement}.
The penalty parameter $t$ is updated as
\begin{equation}\label{eq:update_t}
t_{k+1} := t_k + d_k, 
\end{equation}
where $d_k$ is a decrement or an increment over the current penalty parameter $t_{k}$.
The following lemma shows how we can choose $d_k$; the proof is provided in the Appendix.
%%% Lemma 3.3.
\begin{lemma}\label{le:rel_Delta_lambda}
Let $\Delta_k$ be defined by \eqref{eq:Delta} such that $\Delta_k < 1$ and the penalty parameter for the $(k+1)$-th iteration be updated by \eqref{eq:update_t}. 
Then, we have 
\begin{equation}\label{eq:Delta_est}
\frac{\Delta_k}{1 + \Delta_k} \leq \frac{\abs{d_k}}{t_k}\Vert\nabla^2f(\xb^{*}_{t_{k+1}})\Vert_{\xb^{*}_{t_{k+1}}}^{*} \leq \frac{\abs{d_k}}{t_k}\sqrt{\nu}.
\end{equation}
Consequently, if we choose $d_k$ such that $\abs{d_k} \leq \frac{t_k}{\sqrt{\nu}}$, then $\Delta_k \leq \frac{\abs{d_k}\sqrt{\nu}}{t_k - \abs{d_k}\sqrt{\nu}}$.
\end{lemma}

Now, we combine Lemma \ref{le:rel_Delta_lambda} and Theorem \ref{th:main_statement} to establish an update rule for $t_k$. The condition $\Delta_k \leq \frac{\sqrt{\beta} - 2.581\beta}{2.581 + \sqrt{\beta}} =: C(\beta)$ in Theorem  \ref{th:main_statement} holds if we force
\begin{equation*}
\frac{\abs{d_k}\sqrt{\nu}}{t_k - \abs{d_k}\sqrt{\nu}} \leq C(\beta),
\end{equation*}
which leads to $\abs{d_k} \leq \sigma_{\beta}\cdot t_k$, where $ \sigma_{\beta} := \frac{C(\beta)}{(1 + C(\beta))\sqrt{\nu}} \in (0, 1)$.
%Figure \ref{fig:C_func} shows the behavior of $C(\beta)$. %It is easy to check that $C(\beta)$ attains the maximum at $\beta^{*} \approx 0.03495$, where $C(\beta^{*}) \approx 0.034951$; see 
%\begin{figure}[!ht]
%%\vskip-0.1in
%\begin{center}
%\centerline{\includegraphics[width = 8cm, height = 4.5cm]{C_func}}
%\vskip-0.4cm
%\caption{The behavior of the function $C(\beta)$}\label{fig:C_func}
%\end{center}
%\vskip -0.2in
%\end{figure} 
%Let us fix $\beta \in (0, 0.15]$ where, according to Theorem \ref{th:main_statement}, $\Delta_k \leq C(\beta)$ and define . 
Then, based on Lemma \ref{le:rel_Delta_lambda}, we can update $t_k$ as
\begin{equation}\label{eq:update_tk}
t_{k+1} := (1 \pm \sigma_{\beta})t_k,
\end{equation}
i.e., we can either increase $t_{k}$ or decrease $t_{k}$ by a factor $1 \pm \sigma_{\beta}$ at each iteration while preserving the properties of Lemma \ref{le:rel_Delta_lambda}. 
For example, for $\beta = 0.05 $, we have $C(\beta) \approx 0.033715$ and $\sigma_{\beta} \approx \frac{0.033715}{\sqrt{\nu}}$.

%%%***********************************************************************************
%%% 4. Finding a starting point
%%%***********************************************************************************
\subsection{Finding a starting point}\label{subsec:initial_point}
In order to initialize the path-following phase of our algorithm, we need to find a point $\xb_{t_0} \in \dom{F}$ such that $\lambda_{t_0}^+ := \Vert  \xb_{t_0} - \xb^{*}_{t_0}\Vert_{\xb^{*}_{t_0}} \leq \beta$ for given $\beta \in (0, 0.15]$ as indicated in Theorem \ref{th:main_statement}.
To achieve this goal, we apply the \textit{inexact damped} proximal-Newton method:  
Given $t_0 > 0$ and an initial point $\xb^0\in\dom{F}$, we generate a sequence $\{\xb^j_{t_0}\}_{j\geq 0}$, starting from $\xb_{t_0}^0 := \xb^0$, by computing
\begin{equation}\label{eq:damped_PN_iter}
\xb^{j+1}_{t_0} := \xb^j_{t_0} + \alpha_j\db^j_{t_0}, ~\textrm{with}~ \db^j_{t_0} := \sb^j_{t_0} - \xb^j_{t_0}, ~j\geq 0,
\end{equation}
where $\alpha_j\in (0, 1]$ is a given step size which will be defined later, $\db^j_{t_0}$ is a $\delta_0^j$-approximate proximal-Newton search direction, and $\sb^j_{t_0}$ is a trial point obtained by approximately solving the following convex subproblem:
\begin{align}\label{eq:cvx_subprob2}
\sb^{j}_{t_0} \approx \bar{\sb}^{j}_{t_0} := \argmin_{\sb\in\dom{F}} \widehat{F}(\sb; {t_0}, \sb_{t_0}^{j-1}).
\end{align}
Here we use the index $j$ to distinguish it from the index $k$ of the path-following phase in the previous subsections.
Again, we denote with $\bar{\sb}^{j+1}_{t_0}$ the exact solution of \eqref{eq:cvx_subprob2} and the approximation ``$\approx$'' is defined as in Definition \ref{de:inexact_sol} with the accuracy $\delta_0^j\geq 0$. 

It follows from \eqref{eq:approx_sol} that
\begin{equation}\label{eq:inexact_criterion}
\widehat{F}(\bar{\sb}^j_{t_0}; {t_0}, \sb_{t_0}^{j-1}) \leq \widehat{F}(\sb^j_{t_0}; {t_0}, \sb_{t_0}^{j-1})  \leq \widehat{F}(\bar{\sb}^j_{t_0}; {t_0}, \sb_{t_0}^{j-1}) + \frac{(\delta^j_0)^2}{2}.
\end{equation}
Given the inexact proximal-Newton search direction $\db^j_{t_0}$, we define $\zeta_j := \Vert\db^{j}_{t_0} \Vert_{\xb^j_{t_0}}$ the \textit{inexact proximal-Newton} decrement, which is similar to the proximal-Newton decrement $\bar{\zeta}_j := \Vert\bar{\db}^j_{t_0}\Vert_{\xb^j_{t_0}}$ defined in \cite{TranDinh2013c}. 
The following lemma shows how to choose the step size $\alpha_j$; the proof is given in the appendix.
 
%% Lemma 4.3.
\begin{lemma}\label{le:damped_PN_scheme}
Let $\{\xb^j_{t_0}\}_{j\geq 0}$ be  a sequence generated by the inexact damped proximal-Newton scheme \eqref{eq:damped_PN_iter}. 
If we choose the accuracy $\delta_0^j$ such that $\delta_0^j < \zeta_j$ then, with $\alpha_j := \frac{\zeta_j - \delta_0^j}{(1 + \zeta_j - \delta_0^j)\zeta_j} \in [0, 1]$ we have
\begin{equation}\label{eq:damped_PN_scheme}
F(\xb^{j+1}_{t_0}; t_0) \leq  F(\xb^j_{t_0}; t_0) - \omega(\zeta_j - \delta_0^j), ~~\forall j\geq 0.
\end{equation}
Moreover, the above step size $\alpha_j$ is optimal.
\end{lemma}

At each iteration, assume $\delta_0^j := \kappa \zeta_j$ where $\kappa \in (0, 1)$ (see Section \ref{sec:application}).
For a given $\beta \in (0, 0.15]$, from Lemma \ref{le:lower_bound} we deduce that $\lambda_{t_0}^+ := \norm{\xb_{t_0} - \xb^{*}_{t_0}}_{\xb^{*}_{t_0}} \leq \beta$ if $F(\xb_{t_0}; t_0) - F(\xb^{*}_{t_0}; t_0) \leq \omega(\beta)$. 
To achieve such bound, assume that we can estimate an upper bound of the quantity $\gamma_0 \geq F(\xb^0; t_0) - F(\xb^{*}_{t_0}; t_0) \geq 0$.  
By using the estimate \eqref{eq:damped_PN_scheme}, we deduce
\begin{equation}\label{eq:condition1}
\begin{array}{ll}
F(\xb^{j+1}_{t_0}; t_0)  - F(\xb^{*}_{t_0}; t_0) &\leq F(\xb^0_{t_0}; t_0) - F(\xb^{*}_{t_0}; t_0)  - \sum_{l=0}^{j}\omega\left((1-\kappa)\zeta_l\right) \\
& \leq  \gamma_0 -  \sum_{l=0}^{j}\omega\left((1-\kappa)\zeta_l\right).
\end{array}
\end{equation} 
We now define
\begin{equation}\label{eq:damped_step_cond}
\Gamma_{j_{\max}} := \sum_{l=0}^{j_{\max}}\omega\left((1-\kappa)\zeta_l\right) + \omega(\beta).
\end{equation} 
Then, we observe that, if $\Gamma_{j_{\max}} \geq \gamma_0$, then one can guarantee $\lambda_{t_0}^+  \leq \beta$, where $\xb_{t_0} := \xb^{j_{\max}+1}_{t_0}$.
We note that for $j \leq j_{\max}$, we have $\zeta_j \geq \beta$. Therefore, if we choose $\delta_0^j  := \kappa\beta$ for some $\kappa \in (0, 1)$, then the assumption $\delta_0^j < \zeta_j$ is fulfilled.

% 4.4. Our prototype scheme.
\subsection{Our prototype scheme}\label{sec:phase1_analysis}
The proposed algorithm is given in Algorithm \ref{alg:PF_PN_alg} below.
%%%%%%%%%%%%%%
%%% The algorithm.
%%%%%%%%%%%%%%
\begin{algorithm}[!ht]
   \caption{\textsc{\textbf{Inexact path following proximal Newton algorithm}}}\label{alg:PF_PN_alg}
\begin{algorithmic}[1]
   \Statex {\bfseries Input:} Choose $t_0 > 0$, $\beta \in (0, 0.15]$, $\kappa \in (0, 1)$ and $\xb^0\in\dom{F}$. 
   Compute {~~~~~~}an upper bound $\gamma_0 > 0$ for $F(\xb^0; t_0) - F(\xb^{*}_{t_0}; t_0)$.
   \Statex {\bfseries Initialize:} $\Gamma_{-1} := \omega(\beta)$, $C(\beta) := \frac{\sqrt{\beta} - 2.581\beta}{2.581 + \sqrt{\beta}}$, $\sigma_{\beta} := \frac{C(\beta)}{(1 + C(\beta))\sqrt{\nu}}$, $\xb^0_{t_0} := \xb^0$.
   \Statex\rule{0.88\textwidth}{0.1mm}
    \Statex \centerline{\textsc{\textbf{Phase I: Computing an initial point $\xb^0_{t_0}$}}}
    \vskip-0.6cm
   \Statex\rule{0.88\textwidth}{0.1mm}
   \Statex {\bfseries for} $j = 0, \cdots, j_{\max}$
   \State \hspace{0.16cm} Compute $\mathbf{d}^j_{t_0}$ via \eqref{eq:damped_PN_iter} by solving \eqref{eq:cvx_subprob2} approximately up to $\delta_0^j := \kappa\beta$.
   \State \hspace{0.16cm} Compute $\zeta_j := \Vert\db^j_{t_0}\Vert_{\xb^j_{t_0}}$.
   \State \hspace{0.16cm} $\Gamma_j := \Gamma_{j-1} + \omega((1-\kappa)\zeta_j)$.
   \Statex \hspace{0.16cm} {\bfseries if} $\Gamma_j \geq \gamma_0$ {\bfseries then} 
   \State \hspace{0.48cm} $\xb_{t_0} := \xb^j_{t_0}$.
   \State \hspace{0.48cm} \textbf{break} 
   \Statex \hspace{0.16cm} {\bfseries end if}
   \State \hspace{0.16cm} $\xb^{j+1}_{t_0} := \xb^j_{t_0} + \alpha_j\db^j_{t_0}$, where $\alpha_j := (1-\kappa)\left[1 + (1-\kappa)\zeta_j\right]^{-1}$.
   \Statex {\bfseries end for}
    \Statex\rule{0.88\textwidth}{0.1mm}
    \Statex \centerline{\textsc{\textbf{Phase II: Path following iteration}}}
    \vskip-0.6cm
   \Statex\rule{0.88\textwidth}{0.1mm}
   \State Use $\xb_{t_0}$, computed at \textbf{Phase I}.
   \Statex {\bfseries for} $k = 0, \dots, k_{\max}$ or {\bfseries while} stopping criterion is not met
   \State \hspace{0.16cm} $t_{k+1} := \left(1 \pm \sigma_{\beta}\right)t_k$.  
   \State \hspace{0.16cm} Given $\xb_{t_k}$, solve \eqref{eq:barrier_cvx_subprob2} approximately up to $\delta_k \leq 0.075\beta$ to obtain $\xb_{t_{k+1}} $.
   \Statex {\bfseries end}
\end{algorithmic}
\end{algorithm}
% End of the algorithm.
The main steps are Step 1 and Step 9, where we need to solve two convex subproblems of the form \eqref{eq:barrier_cvx_subprob1}-\eqref{eq:cvx_subprob2}.
For certain regularizers $g$ such as the $\ell_1$-norm or the indicator of a simple convex set, there exist several efficient algorithms for this kind of optimization problems \cite{Beck2009,Becker2011b,Nesterov2004,Nesterov2007}. 
The update rule for $t_k$ at Step 8 of Phase II is based on the worst-case estimate \eqref{eq:Delta_est}. We can add the condition: \textit{if $t_f \in (t_{k-1}, t_k]$, then set $t_k := t_f$} at Step 8 to ensure that the final value of $t_k$ is $t_k = t_f$.
In practice, we can adaptively update $t_k$ as discussed later in Section \ref{sec:application}.
The stopping criterion of Phase II has not been specified yet, which depends on applications as we will see later in Section \ref{sec:constrained_case}.

\subsection{Convergence analysis}\label{sec:alg_conv}
In this subsection, we provide the full complexity analysis for Phase I and Phase II of Algorithm \ref{alg:PF_PN_alg} separately. Since we consider the case $t \downarrow 0^{+}$, we assume $t_{k+1} = (1-\sigma_{\beta})t_k$ and $t_0 \gg 0^{+}$.
The worst-case complexity estimate of Algorithm \ref{alg:PF_PN_alg} is given in the following theorem.

% Theorem 4.1.
\begin{theorem}\label{th:complexity_estimate}
The number of iterations required in Phase I to find $\xb_{t_0}\in\dom{F}$ such that $\lambda_{t_0}^+ \leq \beta$ is at most
\begin{equation}\label{eq:x0_complexity}
j_{\max} := \left\lfloor \frac{F(\xb^0; t_0) - F(\xb^{*}_{t_0}; t_0)}{\omega((1-\kappa)\beta)} \right\rfloor + 1.
\end{equation}
The number of iterations required in Phase II to reach the approximate solution $\xb_{t_f}$ of $\xb^{*}_{t_f}$, where $t_f$ is a user-defined value, close to $0^{+}$ and $\lambda_{t_f}^+ \leq \beta$, is at most
\begin{equation}\label{eq:path_following_complexity}
k_{\max} :=  \left\lfloor \frac{\ln(t_0/t_f)}{-\ln(1 - \sigma_{\beta})} \right\rfloor + 1,
\end{equation}
where $\sigma_{\beta}$ is given by \eqref{eq:update_tk}.
Consequently, the worst-case analytical complexity of Phase II is $\mathcal{O}\left(\sqrt{\nu}\ln\big(\frac{t_0}{t_f}\big)\right)$.
\end{theorem}

% The proof of Theorem 4.1.
\begin{proof}
From Lemma \ref{le:damped_PN_scheme} and the choice of $\delta$ we have 
\begin{equation*}
F(\xb^{j+1}_{t_0}; t_0) - F(\xb^{*}_{t_0}; t_0)\leq  F(\xb^j_{t_0}; t_0) - F(\xb^{*}_{t_0}; t_0) - \omega\left((1-\kappa)\zeta_j\right), ~~\forall j\geq 0.
\end{equation*}
Moreover, by induction, we can show that
\begin{align*}
0 \leq F(\xb^j_{t_0}; t_0) - F(\xb^{*}_{t_0}; t_0) &\leq  F(\xb^0_{t_0}; t_0) - F(\xb^{*}_{t_0}; t_0) - \sum_{l=0}^{j-1}\omega\left((1-\kappa)\zeta_l\right) \nonumber\\
& \leq F(\xb^0_{t_0}; t_0) - F(\xb^{*}_{t_0}; t_0) - j\omega((1-\kappa)\beta).
\end{align*}
This implies 
\begin{align*}
j \leq \frac{F(\xb^0_{t_0}; t_0) - F(\xb^{*}_{t_0}; t_0)}{\omega((1-\kappa)\beta)},
\end{align*}
which shows that the number of iterations to obtain $\lambda_{t_0}^+ \leq \beta$ is at most $j_{\max}$. % given by \eqref{eq:x0_complexity}.

For Phase II,  by induction, we have $t_k = t_0(1-\sigma_{\beta})^k$.
Since we desire $t_{k_{\max}} = t_f$, this leads to ${k_{\max}} \geq \frac{\ln(t_0/t_f)}{-\ln(1-\sigma_{\beta})}$. By rounding up the right-hand side of this inequality, we obtain \eqref{eq:path_following_complexity}.

Finally, note that $\ln(1 - \sigma_{\beta}) \approx \sigma_{\beta}$. By the definition of $\sigma_{\beta} = \frac{C(\beta)}{(C(\beta) + 1)\sqrt{\nu}}$, we can easily observe that the worst-case analytical complexity of Phase II, which is $\mathcal{O}\left(\sqrt{\nu}\ln\big(\frac{t_0}{t_f}\big)\right)$. 
\end{proof}
% End of the proof.

%%%%%%%%%%%%%%%%%%%%%%%%%%%%%%%%%%%%%%%%%%%%%%%%%%%%%%%%
% 5. Application to constrained convex optimization.
%%%%%%%%%%%%%%%%%%%%%%%%%%%%%%%%%%%%%%%%%%%%%%%%%%%%%%%%
\section{Application to constrained convex optimization}\label{sec:constrained_case}
We now specify Algorithm \ref{alg:PF_PN_alg} to solve the constrained convex programming problem of the form \eqref{eq:constr_cvx_prob}.
We assume that $f$ is the $\nu$ - self-concordant barrier associated with $\Omega$ such that $\mathrm{Dom}(f) \equiv \Omega$.
First, we show the relation between the solution of the constrained problem \eqref{eq:constr_cvx_prob} and the parametric problem \eqref{eq:cvx_prob} in the following lemma, whose proof can be found in the appendix.

%% Lemma 2.1.
\begin{lemma}\label{le:barrier_solution}
Let $\xb^{*}$ be a solution of \eqref{eq:constr_cvx_prob} and $\xb^{*}_t$ be the solution of \eqref{eq:cvx_prob} at a given $t > 0$, i.e., $\xb_t^{*} \in \mathrm{int}(\Omega)$. 
Then, for any $t > 0$, $\xb^{*}_t$ is strictly feasible to \eqref{eq:cvx_prob} and
\begin{equation}\label{eq:approx_sol_P_and_CP}
0 \leq g(\xb^{*}_t) - g(\xb^{*}) \leq t\nu.
\end{equation}
Let $\xb_{t_k}$ be the point generated by $\mathrm{Algorithm~\ref{alg:PF_PN_alg}}$ at the $k$-th iteration and $\xb^{*}_{t_k}$ be the solution of \eqref{eq:cvx_prob} at $t = t_k$. Then
\begin{equation}\label{eq:rel_Fk_and_Fstar}
-\nu t_k \leq g(\xb_{t_k}) -  g(\xb^{*}_{t_k}) \leq t_k \left(\sqrt{\nu}\frac{\lambda_{t_k}^+}{1 - \lambda_{t_k}} +  \frac{\lambda_{t_k}}{(1 - \lambda_{t_k})^2}\left(\lambda_{t_k}^+ + \lambda_{t_k} + \delta_k\right) + \frac{\delta_k^2}{2}\right). 
\end{equation}
provided that $\lambda_{t_k} < 1$. 
Consequently, it holds that
\begin{equation}\label{eq:app_sol}
0 \leq g(\xb_{t_k}) - g(\xb^{*}) \leq t_k \psi(\nu, \lambda_{t_k}, \lambda_{t_k}^+,\delta_k),
\end{equation}
where $\psi(\nu, \lambda_{t_k}, \lambda_{t_k}^+, \delta_k) :=  \nu + \sqrt{\nu}\frac{\lambda_{t_k}^+}{1 - \lambda_{t_k}} + \frac{\lambda_{t_k}}{(1-\lambda_{t_k})^2}\left(\lambda_{t_k}^+ + \lambda_{t_k} + \delta_k\right) + \frac{\delta_k^2}{2}$ and $\lambda_{t_k} < 1$.
\end{lemma}

The estimate \eqref{eq:approx_sol_P_and_CP} in Lemma \ref{le:barrier_solution} shows that for sufficiently small $t > 0$, the solution $\xb^{*}_t$ of \eqref{eq:cvx_prob} approximates the solution $\xb^{*}$ of \eqref{eq:constr_cvx_prob}, i.e. $g(\xb^{*}_t) \to g(\xb^{*})$ as $t\downarrow 0^{+}$. 
The estimate \eqref{eq:app_sol} in Lemma \ref{le:barrier_solution}  suggests that if a sequence $\set{ (\xb_{t_k}, t_k)}_{k\geq 0}$  is generated by Algorithm \ref{alg:PF_PN_alg} for $t_f \leq \varepsilon$, then $\set{\xb_{t_k}}_{k\geq 0}$ converges to $\xb^{*}$ provided that the parameter $t_k$ is updated as $t_{k+1} := (1-\sigma_{\beta})t_k $ and $\delta_k \leq \kappa\beta$ (See Theorem \ref{th:main_statement}). 

If we apply Algorithm \ref{alg:PF_PN_alg}  to solve the constrained optimization problem \eqref{eq:constr_cvx_prob}, then we need to change the stopping criterion as $t_f \leq \varepsilon$ for a given accuracy $\varepsilon > 0$. Then the convergence of Algorithm \ref{alg:PF_PN_alg} for solving \eqref{eq:constr_cvx_prob} is given in the following theorem.

%% Theorem 3.4.
\begin{theorem}\label{th:convergence_theorem}
Let $\set{(\xb_{t_k}, t_k)}_{k\geq 0}$ be a sequence generated by Algorithm \ref{alg:PF_PN_alg} for solving \eqref{eq:constr_cvx_prob}. Then, after $\bar{k}$ iterations in Phase II, we have the following bound
\begin{equation}\label{eq:obj_value_est2}
\vert g(\xb_{t_{\bar{k}}}) - g(\xb^{*})\vert \leq \psi(\beta, \nu)t_{\bar{k}},
\end{equation}
where $\psi(\beta,\nu) := \nu + \sqrt{\nu}\frac{\beta}{1 - 0.4\sqrt{\beta}} + \frac{0.4\sqrt{\beta}}{(1 - 0.4\sqrt{\beta})^2}\left(1.075\beta + 0.4\sqrt{\beta}\right) + 0.003\beta^2$ is a constant.

Consequently, the worst-case analytical complexity of Phase II in Algorithm \ref{alg:PF_PN_alg} to achieve an $\varepsilon$-optimal solution, i.e., $\vert g(\xb_{t_{\bar{k}}}) - g(\xb^{*})\vert \leq
\varepsilon$, is $\mathcal{O}\left(\sqrt{\nu}\log(\frac{t_0 \psi(\beta, \nu)}{\varepsilon})\right)$. 
\end{theorem}

% The proof of Theorem 5.2.
\begin{proof}
By the definition of $\psi$ in Lemma \ref{le:barrier_solution} we can easily show that  $\psi(\beta,\nu) \geq \nu$.
On one hand, using this relation and \eqref{eq:app_sol}, we have $\vert g(\xb_{t_k}) - g(\xb^{*})\vert \leq \psi(\beta,\nu)t_{k}$.
On the other hand, by induction, we have $t_k = (1-\sigma_{\beta})^kt_0$ after $k$ iterations. 
Therefore, if $(1-\sigma_{\beta})^k t_0 \psi(\beta, \nu) \leq \varepsilon$, we can conclude that $\vert g(\xb_{t_k}) - g(\xb^{*})\vert \leq \varepsilon$. 
The last condition leads to $k \geq \frac{\log\left(\frac{t_0\psi(\beta, \nu)}{\varepsilon}\right)}{-\log(1-\sigma_{\beta})}$.  
Since $\log(1-\sigma_{\beta}) \approx -\sigma_{\beta}$, we conclude that the worst-case analytical complexity of Phase II in Algorithm \ref{alg:PF_PN_alg} is $\mathcal{O}\left(\sqrt{\nu}\log(\frac{t_0 \psi(\beta, \nu)}{\varepsilon})\right)$.
\end{proof}
% End of the proof.

%%% 6. Applications in graphical model selections and clustering. 
\section{Numerical experiments}\label{sec:application}
In this section, we first discuss the implementation aspects of Algorithm  \ref{alg:PF_PN_alg}. 
Next, we show how to customize this algorithm to solve a standard convex programming problem.
Then, we provide three numerical examples: 
The first example is a synthetic  low-rank approximation problem with additional constraints to highlight the inefficiency of off-the-self IP solvers.
The second one is an application to clustering using max-norm as a concrete example for constrained convex optimization.
The third example is an application to graph learning where we track the approximate solution of this problem along the regularization parameter horizon.

%% **************************************************************************************
%% 6.1. Implementation issue
%% **************************************************************************************
\subsection{Implementation issues}\label{subsec:impl_isueses}
Some fundamental implementation issues in Algorithm \ref{alg:PF_PN_alg} are the following:

% Method for subproblems and warm-start
\paragraph{Methods for subproblems \eqref{eq:barrier_cvx_subprob1} and \eqref{eq:cvx_subprob2} and warm-start}
The main ingredient in Algorithm \ref{alg:PF_PN_alg} is the solution of \eqref{eq:barrier_cvx_subprob1} and \eqref{eq:cvx_subprob2}. 
The more efficiently this problem is solved, the faster Algorithm \ref{alg:PF_PN_alg} becomes. 
For certain classes of $g$, e.g., $\ell_1$-norm, nuclear norm, atomic norm or simple projections, this problem is well-studied. 

Subproblems \eqref{eq:barrier_cvx_subprob1} and \eqref{eq:cvx_subprob2} have the same structure over the iterations.
This observation can be exploited a priori by using the similarity between $\nabla{f}(\xb_{t_{k-1}})$, $\nabla^2f(\xb_{t_{k-1}})$ and $\nabla{f}(\xb_{t_k})$, $\nabla^2f(\xb_{t_k})$, for each $k$.
Since evaluating $\nabla{f}$ and $\nabla^2f$ is the most costly part in the subsolvers, exploiting properly the problem structure for computing these quantities can accelerate the algorithm (see the examples below).

Second, \eqref{eq:barrier_cvx_subprob1} and \eqref{eq:cvx_subprob2} is strongly convex. Several first order methods can be applied and yield a linear convergence \cite{Beck2009,Nesterov2004,Nesterov2007}.
When $g$ is the indicator of a polytope or a convex quadratic set (e.g., Euclidian balls), it turns out to be a quadratic program or a quadratically constrained quadratic program. Efficiency of solving this problem is well-understood.

Finally, warm-start strategies is key for efficiently solving \eqref{eq:barrier_cvx_subprob1} and \eqref{eq:cvx_subprob2}.
Given that the information from the previous iteration is available, the distance bewteen $\xb_{t_k}$ and $\xb_{t_{k+1}}$ is usually small. This observation suggests us to initialize the subsolvers with the solution provided by the previous iteration.
Note that warm-start is very important in active-set methods \cite{Nocedal2006}, which can be used as a workhorse for \eqref{eq:barrier_cvx_subprob1} and \eqref{eq:cvx_subprob2}. 

% Adaptive parameter update.
\paragraph{Adaptive parameter update}
Since the update rule $t_{k+1} := (1 \pm \sigma_{\beta})t_k$ is based on the worst-case estimate of $\sigma_{\beta}$, it is better to replace it by an adaptive factor $\sigma_k$ for acceleration.
In fact, from the proof of Lemma \ref{le:rel_Delta_lambda} we can derive
\begin{equation}\label{eq:adaptive_est1}
t_{k+1}(1 + \Delta_k)^{-1}\Delta_k \leq \abs{d_k}\norm{\nabla{f}(\xb^{*}_{t_k})}^{*}_{\xb^{*}_{t_{k+1}}}.
\end{equation}
First, one can show that $\Vert\nabla{f}(\xb^{*}_{t_k})\Vert^{*}_{\xb^{*}_{t_{k+1}}} \leq (1- \lambda_{t_k})^{-1}\Vert\nabla{f}(\xb^{*}_{t_k})\Vert^{*}_{\xb_{t_k}}$.
Second, by the triangle inequality, we have $\Vert\nabla{f}(\xb^{*}_{t_k})\Vert^{*}_{\xb_{t_k}} \leq \Vert\nabla{f}(\xb^{*}_{t_k}) - \nabla{f}(\xb_{t_k})\Vert^{*}_{\xb_{t_k}} + \Vert\nabla{f}(\xb_{t_k})\Vert^{*}_{\xb_{t_k}}$. 
However, since $\Vert\nabla{f}(\xb^{*}_{t_k}) - \nabla{f}(\xb_{t_k})\Vert^{*}_{\xb_{t_k}} \leq (1-\lambda_{t_k}^+)^{-1}\lambda_{t_k}^+$, the last inequality leads to $\Vert\nabla{f}(\xb^{*}_{t_k})\Vert^{*}_{\xb_{t_k}} \leq (1-\lambda_{t_k}^+)^{-1}\lambda_{t_k}^+ + \Vert\nabla{f}(\xb_{t_k})\Vert^{*}_{\xb_{t_k}}$.
Combining all these derivations, we eventually get 
\begin{equation}\label{eq:adaptive_est2}
\Vert\nabla{f}(\xb^{*}_{t_k})\Vert^{*}_{\xb^{*}_{t_{k+1}}} \leq (1-\lambda_{t_k})^{-1}\left( (1-\lambda_{t_k}^+)^{-1}\lambda_{t_k}^+ + \Vert\nabla{f}(\xb_{t_k})\Vert^{*}_{\xb_{t_k}} \right).
\end{equation}
From Theorem \ref{th:main_statement}, we have  $\lambda_{t_k}^+ \leq\beta$ and $\lambda_{t_k} \leq 0.3874\sqrt{\beta}$. 
Then, if we define
\begin{align}\label{eq:R_k}
R_k(\beta) &:= (1-0.3874\sqrt{\beta})^{-1}\left((1-\beta)^{-1}\beta + \norm{\nabla{f}(\xb_{t_k})}_{\xb_{t_k}}^{*}\right) \nonumber\\
& \leq (1-0.3874\sqrt{\beta})^{-1}\left((1-\beta)^{-1}\beta + \sqrt{\nu}\right),
\end{align}
then, we can derive the update rule for $t_k$ as $t_{k+1} = (1 \pm \sigma_k)t_k$, where $\sigma_k$ is given as
\begin{equation}\label{eq:t_update_new}
\sigma_k := \max\set{ \frac{C(\beta)}{C(\beta) + (1-C(\beta))R_k(\beta)}, \sigma_{\beta}} \in (0, 1),
\end{equation}
and $C(\beta)$ and $\sigma_{\beta}$ are given in the previous section.
A similar strategy for updating $t$ in the case $f(\xb)$ is replaced by $f(\xb) + Q(\xb)$ can be derived by using the same techniques as in \cite{Nesterov2004c}, where $Q$ is a convex quadratic function.

%% **************************************************************************************
%% 6.2. Specification of Algorithm 1 to some standard convex optimization problems.
%% **************************************************************************************
\subsection{Instances of Algorithm \ref{alg:PF_PN_alg}}
Algorithm  \ref{alg:PF_PN_alg} can be customized to solve a broad class of constrained convex problems of the form:
\begin{equation}\label{eq:general_convex}
\left.\begin{array}{ll}
\displaystyle\min_{\xb\in\mathbf{R}^n} &h(\xb) \\
\mathrm{s.t} &\xb \in \mathcal{C}\cap \Omega,
\end{array}\right.
\end{equation}
where $h$ is a proper, lower semicontinuous and convex function, $\mathcal{C}$ is a nonempty, closed and convex set, $\Omega$ is also a nonempty, closed and convex endowed with a $\nu$-self-concordant barrier $f$.
Let $g(\xb) := h(\xb) + \delta_{\mathcal{C}}(\xb)$, where $\delta_{\mathcal{C}}$ is the indicator function of $\mathcal{C}$. 
Then, problem \eqref{eq:general_convex} can equivalently be converted into \eqref{eq:constr_cvx_prob}.

As a concrete example, we show that Algorithm \ref{alg:PF_PN_alg} can be customized to solve the constrained problems of the form \eqref{eq:constr_cvx_prob} with additional linear equality constraints $\mathbf{A}\xb = \mathbf{b}$. For simplicity of discussion, let us consider the following standard quadratic conic programming problem:
\begin{equation}\label{eq:conic}
\left.\begin{array}{ll}
\displaystyle\min_{\xb \in \mathcal{K}} &\frac{1}{2}\xb^T\mathbf{Q}\xb + \mathbf{q}^T\xb\\
\mathrm{s.t.} & \mathbf{A}\xb = \mathbf{b},
\end{array}\right.
\end{equation}
where $\mathbf{Q}$ is a symmetric positive semidefinite and $\mathcal{K}$ is a proper, closed, self-dual cone in $\mathbf{R}^n$ (including positive semidefinite cone), which is endowed with a $\nu$-self-concordant barrier $f$. 
It is also possible to include inequality constraints $\mathbf{B}\xb \leq \mathbf{c}$. 

In order to customize Algorithm \ref{alg:PF_PN_alg} for solving \eqref{eq:conic}, we define $g(\xb) := \frac{1}{2}\xb^T\mathbf{Q}\xb + \mathbf{q}^T\xb + \delta_{\mathcal{C}}(\xb)$, where $\delta_{\mathcal{C}}$ is the indicator function of $\mathcal{C} := \set{\xb\in\mathbf{R}^n ~|~ \mathbf{A}\xb = \mathbf{b}}$. 
Then, problem \eqref{eq:conic} can be cast into  \eqref{eq:constr_cvx_prob}.
In principle, we can apply Algorithm \ref{alg:PF_PN_alg} to solve the resulting problem.
Now, let us consider the corresponding convex subproblem \eqref{eq:barrier_cvx_subprob1} associated with \eqref{eq:conic} as follows
\begin{equation}\label{eq:conic_subprob}
{\small
\begin{array}{ll}
\displaystyle\min_{\xb\in\mathrm{int}(\mathcal{K})} \!\! \Big\{\frac{1}{2}\xb^T\left(t\nabla^2f(\xb_{t_k}) + \mathbf{Q}\right)\xb + \left(\mathbf{q} \!+\! t\nabla{f}(\xb_{t_k}) \!-\! t\nabla^2f(\xb_{t_k})\xb_{t_k}\right)^T\xb + \delta_{\mathcal{C}}(\xb)\Big\}.
\end{array}}
\end{equation}
The optimality condition for this problem becomes
\begin{equation}\label{eq:conic_opt_cond}
\begin{cases}
\left(\mathbf{Q} + t\nabla^2f(\xb_{t_k})\right)\xb + \mathbf{q} + t\nabla{f}(\xb_{t_k}) - t\nabla^2f(\xb_{t_k})\xb_{t_k} + \mathbf{A}^T\mathbf{y} &= 0,\\
\mathbf{A}\xb - \mathbf{b} &= 0.
\end{cases}
\end{equation}
Here, $\mathbf{y}$ is the Lagrange multiplier associated with the equality constraints $\mathbf{A}\xb - \mathbf{b} = 0$.
Let us define $\db := \xb - \xb_{t_k}$, then we can write \eqref{eq:conic_opt_cond} as follows
\begin{equation}\label{eq:linear_system}
\begin{pmatrix} \mathbf{Q} + t\nabla^2f(\xb_{t_k}) & \mathbf{A}^T \\ \mathbf{A} & \mathbf{0}\end{pmatrix}
\begin{pmatrix}\db \\ \yb \end{pmatrix} =
\begin{pmatrix} - \mathbf{q} - \mathbf{Q}\xb_{t_k} - t\nabla{f}(\xb_{t_k}) \\ \mathbf{b} - \mathbf{A}\xb_{t_k}\end{pmatrix}.
\end{equation}
Solving this linear system provides us a Newton search direction for Algorithm \ref{alg:PF_PN_alg}.
In fact, this linear system \eqref{eq:linear_system} coincides with the  system of computing Newton direction in standard primal interior-point methods for solving \eqref{eq:conic} directly,  see, e.g., \cite{BenTal2001,Nesterov1994,Roos2006,Wright1997}.

%% **************************************************************************************
%% 6.4. Simultaneous Sparse and Low-rank matrix completion.
%% **************************************************************************************
\subsection{Low-rank SDP matrix approximation}
To illustrate the scalability and accuracy of the proposed path-following scheme, we consider the following matrix approximation problem, which arises from, e.g., quantum tomography and phase-retrieval \cite{Banaszek1999,Candes2011a}:
\begin{equation}\label{eq:matrix_approx}
\left.\begin{array}{ll}
\displaystyle\min_{\mathbf{X}\in\mathcal{S}^n} & \rho\norm{\vec{\mathbf{X} - \mathbf{M}}}_1 + (1-\rho)\trace{\mathbf{X}}\\
\mathrm{s.t.} & \mathbf{X} \succeq 0, ~\mathbf{L}_{ij} \leq \mathbf{X}_{ij} \leq \mathbf{U}_{ij}, ~i, j = 1,\dots, n.
\end{array}\right.
\end{equation}  
Here $\mathbf{M}\in\mathbb{R}^{n\times n}$ is a given matrix (not necessarily positive definite);  $\rho \in [0, 1]$ is a given regularization parameter and $\mathbf{L}$ and $\mathbf{U}$ are the element-wise lower and upper bound of $\mathbf{M}$.
Problem \eqref{eq:matrix_approx} is a convex relaxation of the problem of approximating $\mathbf{M}$ by a low-rank and positive semidefinite matrix $\mathbf{X}$.
Here, the trace-norm is used to approximate the rank of $\mathbf{X}$ and $\norm{\cdot}_1$ is used to measure the distance from $\mathbf{X}$ to $\mathbf{M}$.

Let $\Omega := \mathcal{S}^n_{++}$ the cone of symmetric positive semidefinite matrices, and $g(\mathbf{X}) := \rho\norm{\vec{\mathbf{X} - \mathbf{M}}}_1 + (1-\rho)\trace{\mathbf{X}} + \delta_{[\mathbf{L}, \mathbf{U}]}(\mathbf{X})$, where $\delta_{[\mathbf{L}, \mathbf{U}]}$ is the indicator function of the interval 
\begin{equation*}
[\mathbf{L}, \mathbf{U}] := \set{\mathbf{X} \in \mathcal{S}^n ~|~ \mathbf{L}_{ij} \leq \mathbf{X}_{ij} \leq \mathbf{U}_{ij}, ~i, j = 1, \cdots, n}.
\end{equation*}
Since $f(\mathbf{X}) := -\log\det(\mathbf{X})$ is  the standard barrier function of $\Omega$, we can reformulate \eqref{eq:matrix_approx} in the form of \eqref{eq:constr_cvx_prob}.

In this example, we  test Algorithm \ref{alg:PF_PN_alg} and compare it with two standard interior point solvers, called \texttt{SDPT3} \cite{Tutunku2003} and \texttt{SeDuMi} \cite{Sturm1999}.
The parameters are configured as follows. We choose $t_0 := 10^{-2}$ and terminate the algorithm if $t_k \leq 10^{-7}$. The starting point $\mathbf{X}^0$ is set to $\mathbf{X}^0 := 0.1\mathbb{I}$, where $\mathbb{I}$ is the identity matrix.
We tackle \eqref{eq:barrier_cvx_subprob2} and \eqref{eq:cvx_subprob2} by applying the FISTA algorithm \cite{Beck2009}, where the accuracy is controlled at each iteration.

The data is generated as follows. 
First, we generate a sparse Gaussian random matrix $\mathbf{R}\sim\mathcal{N}(0,1)$ of the size $n\times k$, where $k = \lfloor 0.25n\rfloor$ is the rank of $\mathbf{R}$, and the sparsity is $25\%$. 
Then, we generate matrix $\mathbf{M} := \mathbf{R}^T\mathbf{R} + 10^{-4}\mathbf{E}$,  where $\mathbf{E}\sim\mathcal{N}(0,\mathbb{I})$. 
The lower bound $\mathbf{L}$ and the upper bound $\mathbf{U}$ are given as $\mathbf{L} := (m_l- 0.1\abs{m_l})\mathbb{I}$ and $\mathbf{U} := (m_u + 0.1\abs{m_u})\mathbb{I}$, where $m_l := \min_{i,j}\mathbf{M}_{ij}$ and $m_u := \max_{i, j}\mathbf{M}_{ij}$. 

We test three algorithms on five problems of size $n \in \lbrace 80, 100, \dots, 160 \rbrace$ w.r.t.\ $\rho = 0.2$. Table \ref{table:matapp_compare} reports the results and the performance of these three algorithms. 
Our platform is \textsc{Matlab} 2011b on a PC Intel Xeon X5690 at 3.47GHz per core with 94Gb RAM. 

\begin{table*}[!ht]
\begin{center}
\newcommand{\cell}[1]{{\!\!\!}#1{\!\!\!}}
\newcommand{\cellr}[1]{{\scriptsize{\!\!\!\!}#1{\!\!\!\!}}}
\caption{Comparison of Algorithm \ref{alg:PF_PN_alg},  \texttt{SDPT3} and \texttt{SeDuMi}}\label{table:matapp_compare}
\begin{footnotesize}
\begin{tabular}{c|c|r|r|r|r|r}\toprule
& \cell{Solver$\backslash n$} & $80$ & $100$ & $120$ & $140$ & $160$\\
\midrule \midrule
\cell{Size} & \cellr{$[n_v; n_c]$} & \cellr{[\textcolor{red}{16,200; 9,720}]} & \cellr{[\textcolor{red}{25,250; 15,150}]} & \cellr{[\textcolor{red}{36,300; 21,780}]} & \cellr{[\textcolor{red}{49,350; 29,610}]} & \cellr{[\textcolor{red}{64,400; 38,640}]} \\
\midrule
\multirow{3}{*}{\cell{Time (sec)}}           & \cell{\texttt{PFPN}}    & \cell{\textcolor{blue}{15.738}} & \cell{\textcolor{blue}{24.046}} & \cell{\textcolor{blue}{24.817}} & \cell{\textcolor{blue}{25.326}} & \cell{\textcolor{blue}{36.531}} \\
                                                              &\cell{\texttt{SDPT3}}   & \cell{156.340} & \cell{508.418} & \cell{881.398} & \cell{1742.502} & \cell{2948.441} \\
                                                              &\cell{\texttt{SeDuMi}}  & \cell{231.530} & \cell{970.390} & \cell{3820.828} & \cell{9258.429} & \cell{17096.580} \\
\midrule 
\multirow{2}{*}{\cell{$g(\mathbf{X}^{\ast})$}} & \cell{\texttt{PFPN}}  & \cell{306.9159} & \cell{\textcolor{blue}{497.6706}} & \cell{\textcolor{blue}{635.4304}} & \cell{\textcolor{blue}{842.4626}} & \cell{\textcolor{blue}{1096.6516}} \\
                                                              & \cell{\texttt{SDPT3}} 	  & \cell{\textcolor{blue}{306.9153}} & \cell{497.6754} & \cell{635.4306} & \cell{842.4644} & \cell{1096.6540} \\
                                                              & \cell{\texttt{SeDuMi}}       & \cell{306.9176} & \cell{497.6821} & \cell{635.4384} & \cell{842.4776} & \cell{1096.6695} \\
\midrule
\multirow{2}{*}{\cell{$[\mathrm{rank}, \mathrm{sparsity}]$} } 
                                                              & \cell{\texttt{PFPN}}    & \cell{$[20, 30.53\%]$} & \cell{$[26, 27.37\%]$} & \cell{$[30, 25.27\%]$} & \cell{$[35, 23.64\%]$} & \cell{$[40, 21.54\%]$} \\
                                                              & \cell{\texttt{SDPT3}}  & \cell{$[20, 41.02\%]$} & \cell{$[25, 36.99\%]$} & \cell{$[30, 51.61\%]$} & \cell{$[35, 45.03\%]$} & \cell{$[40, 49.07\%]$} \\
                                                              & \cell{\texttt{SeDuMi}} & \cell{$[20, 45.23\%]$} & \cell{$[25, 64.20\%]$} & \cell{$[30, 54.83\%]$} & \cell{$[35, 60.87\%]$} & \cell{$[40, 59.24\%]$} \\
\bottomrule
\end{tabular}
\end{footnotesize}
\end{center}
\end{table*} 

From Table \ref{table:matapp_compare} we can see that if we reformulate problem \eqref{eq:matrix_approx} into a standard SDP problem where \texttt{SDPT3} and \texttt{SeDuMi} can solved, then the number of variables $n_v$ and the number of constraints $n_c$ increase rapidly (highlighted with red color).
Consequently, the computational time in \texttt{SDPT3} and \texttt{SeDuMi} also increase significantly compared to Algorithm \ref{alg:PF_PN_alg}. Moreover, \texttt{SeDuMi} is much slower than \texttt{SDPT3} in this particular example.
Since Algorithm \ref{alg:PF_PN_alg}  does not require to transform problem \eqref{eq:matrix_approx} into a standard SDP problem, we can clearly see the computational advantage of this algorithm to standard interior-point solvers, e.g., \texttt{SDPT3} and \texttt{SeDuMi}, for solving problem \eqref{eq:matrix_approx}. 
We note that the performance of Algorithm \ref{alg:PF_PN_alg} can be enhanced by carefully implementing adaptive update strategies for $t_k$,  preconditioning techniques as well as restart tricks for our FISTA procedure.

%% **************************************************************************************
%% 6.3. The max-norm constrained $l_1$-norm optimization problem.
%% **************************************************************************************
\subsection{Max-norm and $\ell_1$-norm optimization in clustering}
In this example, we show an application of Algorithm \ref{alg:PF_PN_alg} to solve  a constrained SDP problem arising from the {correlation} clustering \cite{Bansal2004}, where the number of clusters is unknown. 
Briefly, the problem statement is as follows: Given a graph with $p$ vertices, let $\mathbf{A}$ be its affinity matrix (cf., \cite{Bansal2004} for the definition). The clustering goal here is to partition the set of vertices such that the total disagreement with the edge labels is minimized in $\mathbf{A}$, which is an explicitly combinatorial problem. The work in \cite{Jalali2012} proposes a tight convex relaxation \eqref{eq:clustering_prob}, which poses significant difficulties to the IPM methods when the dimensions scale up. The  approach is called max-norm constrained clustering, and if solved correctly, has rigorous theoretical guarantees of correctness for its solution. 

In this example, we demonstrate that  Algorithm \ref{alg:PF_PN_alg} can obtain medium accuracy solutions in a  scalable fashion as compared to a state-of-the-art IPM.
Here, we use the adaptive update rule \eqref{eq:update_tk}. The algorithm terminates if $t_k \leq 10^{-3}$ and $\lambda_{t_k}^+ \leq 10^{-8}$. We also solve \eqref{eq:barrier_cvx_subprob2} and \eqref{eq:cvx_subprob2} by applying FISTA. 

\begin{table*}[!ht]
 \small
 \newcommand{\cell}[1]{{\!\!}$#1${\!\!}}
\begin{center}
\caption{Average values over 10 Monte Carlo iterations for each dimension $p$.  The variable $\mathbf{K}^{\ast}$ refers to the respective solution at convergence as returned by the algorithms under comparison.}\label{table:compare}
\begin{tabular}{c|r|r|r|r|r|r}\toprule
& $p$ & $50$ & $75$ & $100$ & $150$ & $200$\\
\midrule \midrule
\multirow{3}{*}{Time (sec)}
& PF                & \cell{62.450} & \cell{109.426} & \cell{202.600} & \cell{\textcolor{blue}{416.044}} & \cell{1573.881} \\
& SDPT3             & \cell{\textcolor{blue}{4.396}} & \cell{\textcolor{blue}{21.282}} & \cell{\textcolor{blue}{64.939}} & \cell{522.021} & \cell{2588.721} \\
&\cite{Jalali2012} & \cell{102.217} & \cell{236.366} & \cell{354.444} & \cell{778.904} & \cell{\textcolor{blue}{1420.844}} \\
\midrule 
\multirow{2}{*}{$g(\mathbf{K}^{\ast})$} 
& PF                & \cell{\textcolor{blue}{549.1567}} & \cell{\textcolor{blue}{1293.6727}} & \cell{\textcolor{blue}{2232.5897}} & \cell{\textcolor{blue}{5396.0485}} & \cell{\textcolor{blue}{9809.6066}} \\
& SDPT3             & \cell{549.1860} & \cell{1293.7890} & \cell{2233.0747} & \cell{5396.7305} & \cell{9809.6934} \\
&\cite{Jalali2012} & \cell{597.8825} & \cell{1387.1379} & \cell{2496.6535} & \cell{5583.8605} & \cell{9958.0974} \\
\bottomrule
\end{tabular}
\end{center}
\end{table*} 

We compare our algorithm with the off-the-self, IPM implementation \texttt{SDPT3} \cite{Tutunku2003}, both in terms of time- and memory-complexity. 
Since the curse-of-dimensionality renders the execution of \texttt{SDPT3} impossible in higher dimensions, we use the low precision mode in \texttt{SDPT3} (i.e., $\varepsilon \approx 1.5\times 10^{-8}$) in order to execute larger problems within a reasonable time frame. We compare these two schemes based on synthetic data, generated as described in \cite{Jalali2012}. 

\begin{figure}[!ht]
\begin{minipage}{0.51\textwidth}
\centerline{\includegraphics[width=0.99\linewidth]{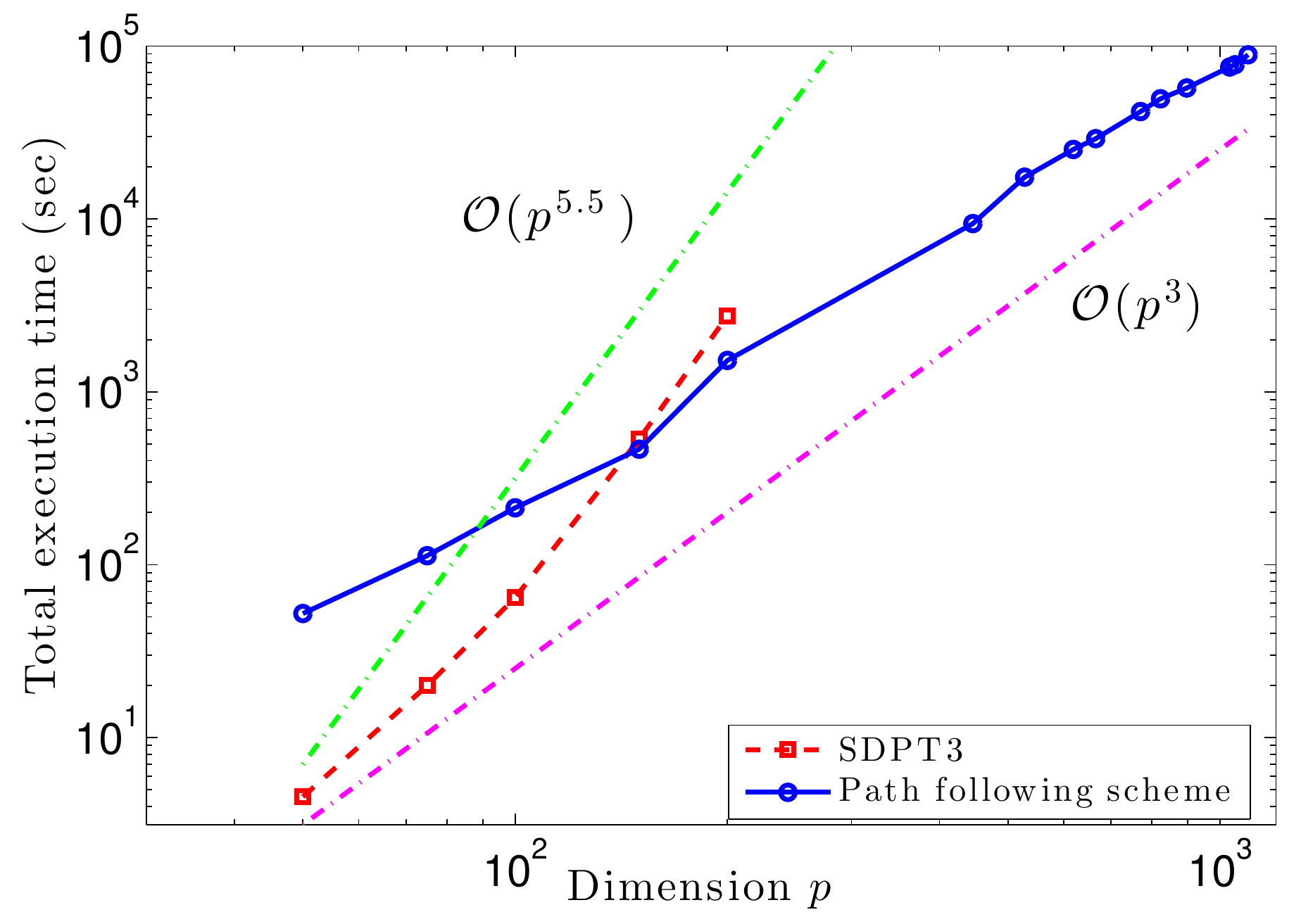}}
\end{minipage}
\begin{minipage}{0.47\textwidth}
\newcommand{\cell}[1]{{\!\!}#1{\!\!}}
\small{\begin{tabular}{c|c|c||c}
& \multicolumn{2}{|c||}{\texttt{SDPT3}} & \cell{PF scheme} \\
\toprule
$p $ & \cell{variables} & \cell{constraints} & \cell{variables} \\
\midrule \midrule
$50$ & $15.1 $ & $2.6 $ & $10$\\
\midrule 
$75$ & $33.9 $ & $5.8 $ & $22.5 $\\
\midrule 
$100$ & $60.2 $ & $10.2 $ & $40$\\
\midrule 
$150$ & $135.3 $ & $22.8 $ & $90 $\\
\midrule 
$200$ & $240.4 $ & $40.4 $ & $160 $\\
\bottomrule
\end{tabular}}
\end{minipage}
\caption{\textbf{(Left)} Execution times. \textbf{(Right)} Number of variables and equality constraints in thousands.  
}\label{fig:sol_compare}
\end{figure} 

In terms of solution accuracy, our scheme with the aforementioned parameter settings is comparable to the low-precision mode of \texttt{SDPT3}, and can often obtain accurate solutions (cf., Table \ref{table:compare}). However, Figure \ref{fig:sol_compare}(Left) illustrates that our path following scheme has a rather dramatic scaling advantage as compared to \texttt{SDPT3}: $\mathcal{O}(p^{3})$ for ours vs. $\mathcal{O}(p^{5.5})$ for \texttt{SDPT3}. Because of this scaling, \texttt{SDPT3} cannot handle problems instances where $p >200$ in our computer. 

Reasons for our scalability are twofold. First, our path following scheme avoids ``lifting'' the problem into higher dimensions. Hence, as the problem dimensions grow (cf., Fig.~\ref{fig:sol_compare}(Right); numbers are in thousands), our memory requirement scales in a better fashion. Moreover, we do not have to handle additional (in)equality constraints.  Second, the subproblem solver has linear convergence rate due to its construction (i.e., $\nabla^2f\succ 0$). Hence, our fast solver (FISTA) obtains medium accuracy solutions quickly since the proximal operator is efficient and has a closed form, and a warm-start strategy is used. 

We also compare the proposed scheme with the scalable Factorization Method (FM), presented in \cite{Jalali2012}: a state-of-the-art, non-convex implementation of \eqref{eq:clustering_prob}, based on splitting techniques. The code is publicly available at \url{http://www.ali-jalali.com/}. We modified this code to include a stopping criterion at a tolerance of 
\begin{equation*}
\norm{\mathbf{K}_{t_{k+1}} - \mathbf{K}_{t_k}}_F \leq 10^{-8}\max\set{\norm{\mathbf{K}_{t_k}}_F, 1}.
\end{equation*}
In Table \ref{table:compare}, we report the average results of $10$ Monte-Carlo realizations for different $p$'s. While the non-convex approach exhibits lower computational complexity empirically,\footnote{Theoretically, FM's computational cost is proportional to the cost of $p\times p$ matrix multiplications. } its solution quality suffers as compared to the convex solution, which has theoretical guarantees. It is clear that the non-convex approach is rather susceptible to local minima.

%% **************************************************************************************
%% 6.2. The Gaussian graphical model selection problem.
%% **************************************************************************************
\subsection{Sparse Pareto frontier in sparse graph learning}\label{subsec:graph_learning}
Many machine learning and signal processing problems naturally feature composite minimization problems where $f$ is directly self-concordant, such as sparse regression with unknown noise variance \cite{Stadler2012}, Poisson imaging \cite{Harmany2012}, one-bit compressive sensing, and graph learning \cite{Ravikumar2011, Kyrillidis2013}. Here, we consider the graph learning problem: Let $\mathbf{\Sigma}$ be the covariance matrix of a Gaussian Markov random field (GMRF) and let $\Xb= \mathbf{\Sigma}^{-1}$. To satisfy the conditional dependencies with respect to the GMRF, $\Xb$ must have zero in $\Xb_{ij}$
corresponding to the absence of an edge between node $i$ and node $j$  \cite{Dempster1972}. Hence, given the empirical covariance  $\widehat{\mathbf{\Sigma}} \succeq 0 $, which is possibly rank deficient, we would like to learn the underlying GMRF. 

It turns out that we can still learn GMRF's with theoretical consistency guarantees from a number of data samples as few as $m = \mathcal{O}(d^2\log p)$ \cite{Ravikumar2011}, where $d$ is the graph node degree, via 
\begin{equation}\label{eq:g_learn}
\min_{\Xb \in \mathbb{R}^{p \times p}:~\Xb\succ 0}\set{-\log\det(\Xb) + \trace{\widehat{\mathbf{\Sigma}}\Xb} + \rho\norm{\vec{\Xb}}_1}, 
\end{equation} 
where $\rho > 0$ is a regularization parameter. We easily observe that \eqref{eq:g_learn} satisfies the $\mathcal{P}(t)$ formulation for $t = 1/\rho$. 
Unfortunately, the theoretical results only indicate the existence of a regularization parameter for consistent estimates and we have to tune to obtain the best  $\rho^*$ in practice. 
We note that the function $f(\Xb) := -\log\det(\Xb)$ is a self-concordant barrier of $\mathcal{S}^p_{+}$. 
As discussed in Subsection \ref{subsec:impl_isueses}, we can modify the update rule for $\rho_k$, we can still apply Algorithm \ref{alg:PF_PN_alg} to track the Pareto frontier of problem \eqref{eq:g_learn} for the case $f(\cdot) + \langle{\mathbf{c}, \cdot}\rangle$ instead of $f(\cdot)$, see, e.g., \cite{Nesterov2004c}.

To the best of our knowledge, the selection of $\rho^{\ast}$ with respect to a general-purpose objective, such as $\mathcal{P}(\rho)$, still remains widely open. For GMRF learning, a homotopy approach is proposed in \cite{Lu2010,Scheinberg2009}, where $\rho$ is updated by a non-adaptive multiplicative factor such that $\rho_{k+1} = c \rho_k$ for $0 < c < 1$. This approach is usually time consuming in practice, and may skip solutions with sparsity close to the desired sparsity level.
Traditionally, \eqref{eq:g_learn}  is addressed by IPM's. Other than \cite{Tran-Dinh2013b} exploited here, we do not know any scalable method that has rigorous global convergence guarantees for \eqref{eq:g_learn} as it has a globally non-Lipschitz continuous gradient. 
The authors in \cite{Banerjee2008} use a probabilistic heuristic to select $\rho$: as the number of samples go to infinity, this heuristic leads to the maximum likelihood (unregularized) estimator. In practice though, the proposed $\rho$ values are quite large and do not consistently lead to good solutions.

\begin{figure}[t]
\centerline{\includegraphics[width=.99\linewidth, height = 4.2cm]{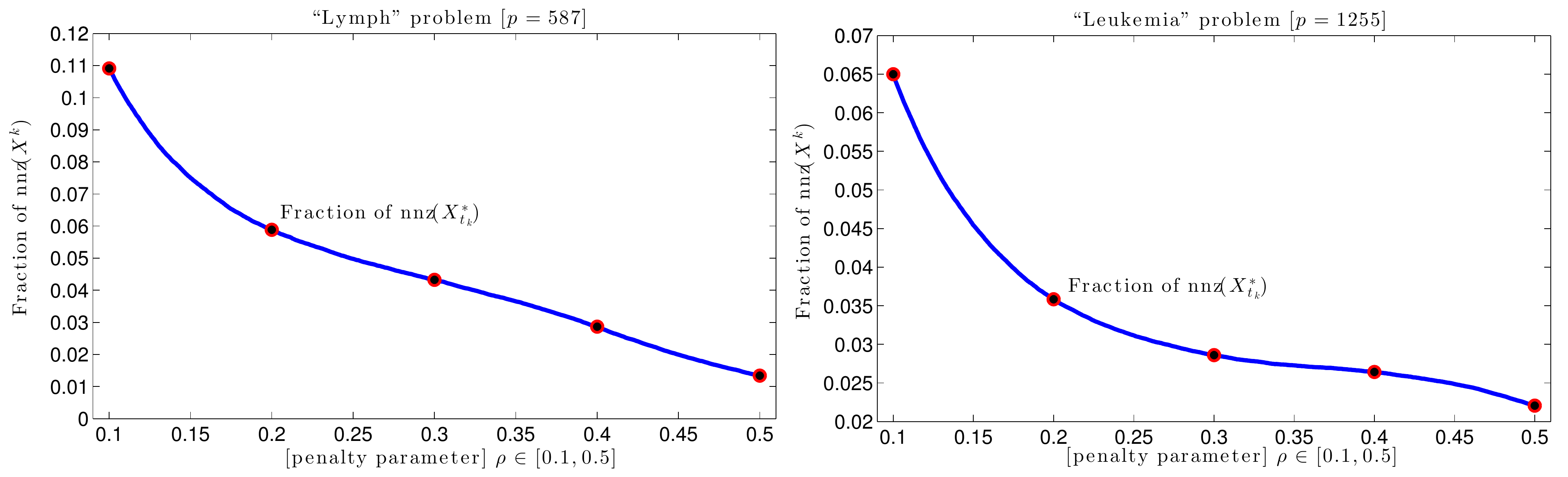}}
\caption{Impact of the regularization parameter to the solution sparsity.}\label{fig:glearn_compare3}
\end{figure} 
 
To this end, our scheme provides an adaptive strategy on how to update the regularization parameter. For instance, we can pick a range $\rho \in [\rho_{\text{min}}, \rho_0]$ and apply our path-following scheme, starting from $\rho_0$ until we either achieve the desired solution sparsity or we reach the lower bound $\rho_{\text{min}}$.  To illustrate the approach, we choose two real data examples from \url{http://ima.umn.edu/~maxxa007/send_SICS/}: \texttt{Lymph} and \texttt{Leukemia}, where the GMRF sizes are $p=587$ and $1255$, respectively.   Figure \ref{fig:glearn_compare3} shows the solution sparsity vs.\ the penalty parameter curve (not to be confused with $f$ vs.\ $g$ curve, which is a convex Pareto curve) as obtained in a tuning-free fashion by our  scheme. 

In order to verify the obtained Pareto curve $\set{\mathbf{X}^k_{\rho_k}}$ well approximates the true solution trajectory $\mathbf{X}^{*}(\rho)$ of the problem \eqref{eq:g_learn}, we apply the proximal-Newton algorithm in \cite{TranDinh2013c} to compute the approximate solution $\widetilde{\mathbf{X}}(\rho_k)$ to $\mathbf{X}^{*}(\rho_k)$ at five different points of $\rho$.
The relative errors $e_k := \norm{\mathbf{X}_{\rho_k} - \widetilde{\mathbf{X}}(\rho_k)}_F/\max\set{\widetilde{\mathbf{X}}(\rho_k)}$ as well as the number of nonzero elements \texttt{n.n.z.} are shown in Table \ref{tbl:rel_errors}.
\begin{table}
\newcommand{\cell}[1]{{\!\!\!}#1{\!\!\!}}
\begin{center}
\caption{The relative error and the number of nonzero elements of two approximate solutions}\label{tbl:rel_errors}
\begin{footnotesize}
\begin{tabular}{c|c|c|c|c|c}\toprule
$\rho$ & 0.1 & 0.2 & 0.3 & 0.4 & 0.5 \\
\midrule \midrule
\multicolumn{6}{c}{Lymph $(n = 587)$} \\ \midrule
\cell{Relative error $e_k$} & \cell{0.0011} & \cell{0.0013} & \cell{0.0018} & \cell{0.0018} & \cell{$7.5342\times 10^{-6}$}\\ \midrule
\cell{\texttt{n.n.z.} ($\widetilde{\mathbf{X}}(\rho_k)/\mathbf{X}_{\rho_k}$)} & \cell{37587/37561} & \cell{20275/20269} & \cell{14901/14875} & \cell{9869/9871} & \cell{4615/4615} \\ \midrule
\multicolumn{6}{c}{Leukemia $(n = 1255)$} \\ \midrule
\cell{Relative error $e_k$} & \cell{$6.1643\times 10^{-4}$} & \cell{$5.5701 \times 10^{-4}$} &  \cell{$6.2124 \times 10^{-4}$} &  \cell{$5.6060\times 10^{-4}$} & \cell{$3.6497\times 10^{-6}$}\\ \midrule
\cell{\texttt{n.n.z.} ($\widetilde{\mathbf{X}}(\rho_k)/\mathbf{X}_{\rho_k}$)}  & \cell{102313/102253} & \cell{56451/56421} & \cell{45051/45055} & \cell{41613/41609} & \cell{34761/34761}\\ 
\bottomrule
\end{tabular}
\end{footnotesize}
\end{center}
\end{table}
We can see from this table that both solutions are relatively close to each other both in terms of relative error and the sparsity.
\section{Concluding remarks}\label{sec:concluding_remarks}
We have proposed a new inexact path-following framework for minimizing (possibly) non-smooth and non-Lipschitz gradient objectives under constraints that admit a self-concordant barrier. 
We have shown how to solve such problems without lifting problem dimensions via additional slack variables and constraints.
Our method is quite modular: custom implementations only require the corresponding custom solver for the composite subproblem \eqref{eq:barrier_cvx_subprob2} with a strongly convex quadratic smooth term and a tractable proximity of the second term $g$.
We have provided a rigorous analysis that establish the worst-case analytical complexity of our approach via a new joint treatment of proximal methods and self-concordant optimization schemes. 
While our scheme maintains the original problem structure, its  worst-case analytical complexity of the outer loop remains the same as in standard path-following interior point methods \cite{Nesterov2004}.  
However, the overall complexity of the algorithm depends on the solution of the subproblem \eqref{eq:barrier_cvx_subprob2}.
We have also shown how the new scheme can obtain points on the Pareto frontier of regularized problems (with globally non-Lipschitz gradient of the smooth part $f$).  
We have numerically illustrated our method on three examples involving the nonsmooth constrained convex programming problems of matrix variables. 
Numerical results have shown that the new path-following scheme is superior to some off-the-self IP solvers such as \texttt{SeDuMi} and \texttt{SDPT3} that require to transform the problem into standard conic programs.

\section{Acknowledgment}
The authors are also grateful to the anonymous reviewers as well as to the associate editor for their thorough comments and suggestions on improving the content and the presentation of this paper.
This work was supported in part by the European Commission under Grant MIRG-268398, ERC Future Proof and SNF 200021-132548, SNF 200021-146750 and SNF CRSII2-147633.
%%%%%%%%%%%%%%%%%%%%%%%%%%%%%%%%%%%%%%%%%%%%%%%%%%%%%%
%%%                           				Supplementary material.
%%%%%%%%%%%%%%%%%%%%%%%%%%%%%%%%%%%%%%%%%%%%%%%%%%%%%%
%\newpage
\appendix
\section{Technical proofs}\label{app:tech_proofs}
We provide in this appendix the full proofs of two theorems: Theorem \ref{th:quad_converg_FPNM} and Theorem \ref{th:main_statement}, and three technical lemmas: Lemma \ref{le:rel_Delta_lambda}, Lemma \ref{le:damped_PN_scheme} and Lemma \ref{le:barrier_solution}.

%% Proof of Theorem 3.2.
\subsection{The proof of Theorem \ref{th:quad_converg_FPNM}}
%\begin{proof}
We define the restricted approximate gap between $\nabla^2f(\xb^{*}_t)$ and $\nabla^2f(\xb_t)$ along the direction $\bar{\xb}^+_t - \xb_t$ as $\bar{\mathbf{r}}_t := (\nabla^2f(\xb^{*}_t) - \nabla^2f(\xb_t))(\bar{\xb}^+_t - \xb_t)$.
Then, by using the definition \eqref{eq:P_x} of $P_{\xb}^g$ and \eqref{eq:S_x} of $S_{\xb}$, we can write \eqref{eq:opt_cond_subprob} equivalently to
\begin{equation}\label{eq:opt_cond_subprob1}
\bar{\xb}^+_t = P^g_{\xb^{*}_t}\left( S_{\xb^{*}_t}(\xb_t) + \bar{\mathbf{r}}_t; t\right).
\end{equation}
Now, we can estimate $\bar{\lambda}^+_t := \norm{\bar{\xb}^+_t - \xb^{*}_t}_{\xb^{*}_t}$ as follows
\begin{align}\label{eq:th32_est_1}
\bar{\lambda}^+_t &:= \norm{\bar{\xb}^+_t - \xb^{*}_t}_{\xb^{*}_t} \nonumber\\
&\overset{\tiny\eqref{eq:opt_cond_subprob1}+\eqref{eq:fixed_point_xstar}}{=} \norm{P^g_{\xb^{*}_t}\left( S_{\xb^{*}_t}(\xb_t) + \bar{\mathbf{r}}_t; t\right) - P^g_{\xb^{*}_t}\left(S_{\xb^{*}_t}(\xb^{*}_t); t\right) }_{\xb^{*}_t} \nonumber\\
&\overset{\tiny\eqref{eq:nonexapansiveness}}{\leq} \norm{ S_{\xb^{*}_t}(\xb_t)  - S_{\xb^{*}_t}(\xb^{*}_t) + \bar{\mathbf{r}}_t }_{\xb^{*}_t}^{*} \nonumber\\
&\leq \norm{ S_{\xb^{*}_t}(\xb_t)  - S_{\xb^{*}_t}(\xb^{*}_t)}_{\xb^{*}_t}^{*}  + \norm{\bar{\mathbf{r}}_t }_{\xb^{*}_t}^{*}.
\end{align}
Similarly to the proof of \cite[Theorem 4.1.14]{Nesterov2004} or \cite[Theorem 5]{TranDinh2013c}, we show that
\begin{equation}\label{eq:th32_est_2}
\norm{ S_{\xb^{*}_t}(\xb_t)  - S_{\xb^{*}_t}(\xb^{*}_t)}_{\xb^{*}_t}^{*}  \leq \frac{\lambda_t^2}{1 - \lambda_t},
\end{equation}
provided that $\lambda_t < 1$.

Next, we estimate $\norm{\bar{\mathbf{r}}_t}_{\xb^{*}_t}^{*}$.
We have
\begin{align}\label{eq:th32_est_3}
 \norm{\bar{\mathbf{r}}_t}_{\xb^{*}_t}^{*} &=  \big\Vert\big(\nabla^2{f}(\xb^{*}_t) - \nabla^2f(\xb_t)\big)(\bar{\xb}^+_t - \xb_t)\big\Vert_{\xb^{*}_t}^\ast\nonumber \\ 
 &\leq \norm{\nabla^2 f(\xb_{t}^{\ast})^{-1/2} \big(\nabla^2{f}(\xb^{*}_t) - \nabla^2f(\xb_t)\big) \nabla^2 f(\xb_{t}^{\ast})^{-1/2}}_{2\rightarrow 2} \norm{\bar{\xb}^+_t - \xb_t}_{\xb^{\ast}_{t}} \nonumber \\
&= \norm{\mathbb{I} -  \nabla^2 f(\xb_{t}^{\ast})^{-1/2} \nabla^2f(\xb_t)\nabla^2 f(\xb_{t}^{\ast})^{-1/2}}_{2\rightarrow 2} \norm{\bar{\xb}^+_t - \xb_t}_{\xb^{\ast}_{t}},
\end{align}
where $\norm{\cdot}_{2\to2}$ is the $\ell_2$-norm of a matrix, i.e., $\norm{\Xb}_{2\to 2} := \max_{\ub}\set{\norm{\Xb\ub}_2 ~|~ \norm{\ub}_2 = 1}$ for a given matrix $\Xb$.
By applying \cite[Theorem 4.1.6]{Nesterov2004}, we can show that 
\begin{align}
\Big\Vert\mathbb{I} -  \nabla^2 f(\xb_{t}^{\ast})^{-1/2} \nabla^2f(\xb_t)\nabla^2 f(\xb_{t}^{\ast})^{-1/2}\Big\Vert_{2\rightarrow 2}  &\leq \max\set{ 1 - (1 - \lambda_t)^2, (1-\lambda_t)^{-2}-1} \nonumber \\ 
&= \frac{2\lambda_t- \lambda_t^2}{(1 - \lambda_t)^2}. \nonumber
\end{align} 
Substituting this estimate into \eqref{eq:th32_est_3} and then using the triangle inequality, we obtain
\begin{equation}\label{eq:th32_est_4}
\norm{\bar{\mathbf{r}}_t}_{\xb^{*}_t} \leq \left(\frac{2\lambda_t- \lambda_t^2}{(1 - \lambda_t)^2}\right)\norm{\bar{\xb}^+_t - \xb_t}_{\xb^{*}_t} \leq \left(\frac{2\lambda_t - \lambda_t^2}{(1 - \lambda_t)^2}\right)\left(\bar{\lambda}_t^+ + \lambda_t\right),
\end{equation}
provided that $\lambda_t < 1$.

Substituting \eqref{eq:th32_est_2} and \eqref{eq:th32_est_4} into \eqref{eq:th32_est_1} and then rearranging the result, we deduce 
\begin{equation}\label{eq:th32_est_5}
\bar{\lambda}_t^+ \leq \left(\frac{3 - 2\lambda_t}{1 - 4\lambda_t + 2\lambda_t^2}\right)\lambda_t^2,
\end{equation}
provided that $1 - 4\lambda_t + 2\lambda_t^2 > 0$.
We can easily show that the condition $1 - 4\lambda_t + 2\lambda_t^2 > 0$ holds if $\lambda_t \in [0, 1 - \frac{\sqrt{2}}{2})$.

Note that $0\preceq \nabla^2f(\xb^{*}_t) \preceq (1 - \lambda_t)^{-2}\nabla^2f(\xb_t)$ due to \cite[Theorem 4.1.6]{Nesterov2004}. Thus, for any $\ub$, we have $\norm{\ub}_{\xb^{*}_t} \leq (1 - \lambda_t)^{-1}\norm{\ub}_{\xb_t}$.
By using this inequality, \eqref{eq:computable_criterion2} and the triangle inequality, it is easy to show that
\begin{align*}
\lambda_t^+ &= \norm{\xb^+_t - \xb^{*}_t}_{\xb^{*}_t} \leq \norm{\xb^+_t - \bar{\xb}^+_t}_{\xb^{*}_t} + \norm{\bar{\xb}^+_t - \xb^{*}_t}_{\xb^{*}_t} \nonumber \\
&\overset{\tiny\eqref{eq:computable_criterion2}}{\leq} (1 - \lambda_t)^{-1}\delta_k + \bar{\lambda}_t^+.
\end{align*}
By substituting \eqref{eq:th32_est_5} into this inequality, we obtain
\begin{align}\label{eq:th32_est_6}
\lambda_t^+ \leq \frac{\delta_k}{1-\lambda_t} + \left(\frac{3 - 2\lambda_t}{1 - 4\lambda_t + 2\lambda_t^2}\right)\lambda_t^2.
\end{align}
Since $\lambda_t \in [0, 1 - \frac{\sqrt{2}}{2})$, the right-hand side of \eqref{eq:th32_est_6} is well-defined. 
Moreover, it is obvious to check that the right-hand side of  \eqref{eq:th32_est_6}  is increasing w.r.t. $\delta_k \geq 0$ and $\lambda_t \in [0, 1-\frac{\sqrt{2}}{2})$.
\Eproof
%% End of the proof.

\subsection{The proof of Theorem \ref{th:main_statement}}
%%% Proof of Theorem 3.2.
We define the function $\psi(\lambda) := \frac{3 -2\lambda}{1 - 4\lambda + 2\lambda^2}$. It is easy to check that $\psi$ is increasing in $[0, 1 - \sqrt{2}/2)$.
Let us limit the range of $\lambda \in [0, 0.15]$. Then, one can show that $\max\set{\psi(\lambda)  ~|~ \lambda \in [0, 0.15]} \leq 6.07$.  
Hence, we can upper estimate \eqref{eq:main_estimate1} by some elementary calculations to obtain
\begin{equation}\label{eq:main_estimate2}
\lambda_{t_{k+1}}^+  \leq 1.18\delta_k + 6.07\lambda_{t_{k+1}}^2.
\end{equation}
Now, we recall the following estimate from \cite[Lemma A.1.(c)]{TranDinh2012c} as
\begin{align}{\label{eq:quoc}}
\lambda_{t_{k+1}} \leq \frac{\lambda_{t_k}^+ + \Delta_k}{1 - \Delta_k},
\end{align}
provided that $\Delta_k < 1$.

Let us fix some $\beta \in (0,  0.15]$. By the assumption $\lambda_{t_k}^+ \leq \beta$, it follows from \eqref{eq:quoc} that
\begin{equation}\label{eq:th33_proof2b}	
\lambda_{t_{k+1}}  \leq \frac{\lambda_{t_k}^+  + \Delta_k}{1 - \Delta_k} \leq \frac{\beta + \Delta_k}{1-\Delta_k}.
\end{equation}	
Substituting \eqref{eq:th33_proof2b} into \eqref{eq:main_estimate2} we obtain
\begin{equation}\label{eq:th33_proof3}
\lambda_{t_{k+1}}^+ \leq 1.18\delta_k + 6.07\left(\frac{\beta + \Delta_k}{1 - \Delta_k}\right)^2.
\end{equation}
Since we desire $\lambda_{t_{k+1}}^+ \leq \beta$, by using \eqref{eq:th33_proof3}, we require $\left(\frac{\beta + \Delta_k}{1-\Delta_k}\right)^2 \leq \frac{\beta - 1.18\delta_k}{6.07}$ provided that $\delta_k < \beta/1.18$.
Since $\delta_k \leq 0.075\beta$, the last condition leads to
\begin{align}
0 \leq \Delta_k \leq \frac{\sqrt{\beta} - 2.581\beta}{2.581 + \sqrt{\beta}} < 1,
\end{align} 
for any $\beta \in (0, 0.15]$.
Finally, we can easy check that $\lambda_{t_{k+1}} \leq \frac{1}{2.581}\sqrt{\beta}$ for $k\geq 0$ due to \eqref{eq:main_estimate2}. 
\Eproof
% End of the proof.

\subsection{The proof of Lemma \ref{le:rel_Delta_lambda}}
% The proof of Lemma 3.3.
Since $\xb^{*}_{t_k}$ and $\xb^{*}_{t_{k+1}}$ are the solutions of \eqref{eq:cvx_prob} at $t = t_k$ and $t = t_{k+1}$, respectively, they satisfy the following optimality conditions:
\begin{align*}
&\mathbf{0} \in t_k\nabla{f}(\xb^{*}_{t_k}) + \partial{g}(\xb^{*}_{t_k}),\\
&\mathbf{0} \in t_{k+1}\nabla{f}(\xb^{*}_{t_{k+1}}) + \partial{g}(\xb^{*}_{t_{k+1}}).
\end{align*} 
Hence, there exist $\vb_k \in \partial{g}(\xb^{*}_{t_k})$ and $\vb_{k+1} \in \partial{g}(\xb^{*}_{t_{k+1}})$ such that $\vb_k = -t_k\nabla{f}(\xb^{*}_{t_{k}})$ and $\vb_{k+1} = - t_{k+1}\nabla{f}(\xb^{*}_{t_{k+1}})$. 
Then, we have
\begin{align*}
\vb_{k+1} - \vb_k &= t_k\nabla{f}(\xb^{*}_{t_k}) - t_{k+1}\nabla{f}(\xb^{*}_{t_{k+1}})  \nonumber \\
& \overset{\tiny\eqref{eq:update_t}}{=} t_k\left(\nabla{f}(\xb^{*}_{t_k}) - \nabla{f}(\xb^{*}_{t_{k+1}}) \right)  - d_{k}\nabla{f}(\xb^{*}_{t_{k+1}}).
\end{align*}
By using the convexity of $g$, the last expression implies
\begin{align*}
0 & \leq (\vb_{k+1} - \vb_k)^{T}(\xb^{*}_{t_{k+1}} - \xb^{*}_{t_k}) \nonumber \\ 
& =  t_k\left(\nabla{f}(\xb^{*}_{t_k}) - \nabla{f}(\xb^{*}_{t_{k+1}}) \right)^{T}(\xb^{*}_{t_{k+1}} - \xb^{*}_{t_k})  - d_{k}\nabla{f}(\xb^{*}_{t_{k+1}})^{T}(\xb^{*}_{t_{k+1}} - \xb^{*}_{t_k}) \nonumber\\
&\leq t_k\left(\nabla{f}(\xb^{*}_{t_k}) - \nabla{f}(\xb^{*}_{t_{k+1}}) \right)^{T}(\xb^{*}_{t_{k+1}} - \xb^{*}_{t_k}) + \abs{d_{k}}\Vert\nabla{f}(\xb^{*}_{t_{k+1}})\Vert_{\xb^{*}_{t_{k+1}}}^{*}\Vert\xb^{*}_{t_{k+1}}-\xb^{*}_{t_k}\Vert_{\xb^{*}_{t_{k+1}}},
\end{align*} 
where the last inequality is due to the generalized Cauchy-Schwatz inequality. 
Since $t_k > 0$, we further have
\begin{align}\label{eq:lm34_est3}
\left(\nabla{f}(\xb^{*}_{t_{k + 1}})  -  \nabla{f}(\xb^{*}_{t_k})\right)^T (\xb^{*}_{t_{k + 1}}   -  \xb^{*}_{t_k}) \leq \frac{\abs{d_k}}{t_k}\Vert\nabla{f}(\xb^{*}_{t_{k + 1}})\Vert_{\xb^{*}_{t_{k + 1}}}^{*}\Vert\xb^{*}_{t_{k + 1}}  -  \xb^{*}_{t_k}\Vert_{\xb^{*}_{t_{k + 1}}}.
\end{align}
However, since $f$ is standard self-concordant, by applying \cite[Theorem 4.1.7]{Nesterov2004}, we have
\begin{align*}
\left(\nabla{f}(\xb^{*}_{t_{k+1}}) - \nabla{f}(\xb^{*}_{t_k})\right)^T(\xb^{*}_{t_{k+1}} - \xb^{*}_{t_k}) \geq \frac{\Vert\xb^{*}_{t_{k+1}} - \xb^{*}_{t_k}\Vert_{\xb^{*}_{t_{k+1}}}^{2}}{1 + \Vert\xb^{*}_{t_{k+1}} - \xb^{*}_{t_k}\Vert_{\xb^{*}_{t_{k+1}}}}.
\end{align*} 
Using this inequality together with  \eqref{eq:lm34_est3} we obtain
\begin{align*}
\frac{\Vert \xb^{*}_{t_{k+1}} - \xb^{*}_{t_k}\Vert_{\xb^{*}_{t_{k+1}}}}{1 + \Vert \xb^{*}_{t_{k+1}} - \xb^{*}_{t_k}\Vert_{\xb^{*}_{t_{k+1}}}} &\leq
\frac{\abs{d_k}}{t_k}\Vert \nabla{f}(\xb^{*}_{t_{k+1}})\Vert_{\xb^{*}_{t_{k+1}}}^{*} \overset{\tiny\eqref{eq:used}}{\leq} \frac{\abs{d_k}}{t_k} \sqrt{\nu}.
\end{align*} 
where by the definition of $\Delta_k$, this completes the proof of \eqref{eq:Delta_est}. The last statement in Lemma \ref{le:rel_Delta_lambda} is a direct consequence of \eqref{eq:Delta_est}.
\Eproof
% End of the proof.

% The proof of Lemma 4.3.
\subsection{The proof of Lemma \ref{le:damped_PN_scheme}}
%\begin{proof}
Let $g_0(\cdot) := t^{-1}_0g(\cdot)$.
Similar to the proof of \cite[Lemma 3.3]{TranDinh2013c}, we can estimate
\begin{small}
\begin{equation}\label{eq:lm43_est1}
F(\xb^{j+1}_{t_0}; t_0)  -  F(\xb^j_{t_0}; t_0)  \leq  -\alpha_j\nabla{f}(\xb^j_{t_0})^T\db^j  +  \omega_{*}(\alpha_j\zeta_j)  +  \alpha_j\left(g_0(\sb^j_{t_0})  -  g_0(\xb^j_{t_0})\right),
\end{equation}
\end{small}
where $\alpha_j\zeta_j < 1$ and $\omega_{*}(\tau) := -\tau - \ln(1 - \tau)$.
From the definition \eqref{eq:F_x} of $\widehat{F}$ and \eqref{eq:inexact_criterion} we have 
\begin{align}\label{eq:lm43_est2}
g_0(\sb^j_{t_0}) - g_0(\xb^j_{t_0}) &\leq g_0(\bar{\sb}^j_{t_0}) - g_0(\xb^j_{t_0}) + \frac{(\delta_0^j)^2}{2} + \nabla{f}(\xb^j_{t_0})^T(\bar{\sb}^j_{t_0} - \sb^j_{t_0}) \nonumber\\
& + \frac{1}{2}\left(\norm{\bar{\sb}^j - \xb^j}_{\xb^j}^2 - \norm{\sb^j - \xb^j}_{\xb^j}^2\right).
\end{align}
Since $\bar{\sb}^j_{t_0}$ is the exact solution of \eqref{eq:cvx_subprob2}, using the optimality condition \eqref{eq:opt_cond_subprob} of this problem, we have
\begin{equation}\label{eq:lm43_est3}
\bar{\vb}^j_{t_0} = -\nabla{f}(\xb^j_{t_0}) - \nabla^2f(\xb^j_{t_0})(\bar{\sb}^j_{t_0} - \xb^j_{t_0}), ~ \bar{\vb}^j_{t_0} \in \partial{g}_0(\bar{\sb}^j_{t_0}).
\end{equation}
By the convexity of $g_0$, \eqref{eq:lm43_est3} implies
\begin{equation*}
g_0(\bar{\sb}^j_{t_0}) - g_0(\xb^j_{t_0}) \leq -\nabla{f}(\xb^j_{t_0})^T(\bar{\sb}^j_{t_0} - \xb^j_{t_0}) - \Vert\bar{\sb}^j_{t_0} - \xb^j_{t_0}\Vert_{\xb^j_{t_0}}^2.
\end{equation*}
Substituting this inequality into \eqref{eq:lm43_est2} and rearranging the result by using $\zeta_j = \Vert \db^j_{t_0}\Vert_{\xb^j_{t_0}} = \Vert \sb^j_{t_0} - \xb^j_{t_0}\Vert_{\xb^j_{t_0}}$, we obtain
\begin{align}\label{eq:lm43_est4}
g_0(\sb^j_{t_0}) - g_0(\xb^j_{t_0}) \leq  \frac{(\delta^j_0)^2}{2} - \nabla{f}(\xb^j_{t_0})^T(\sb^j_{t_0} - \xb^j_{t_0}) - \frac{1}{2}\left(\Vert\bar{\sb}^j_{t_0} - \xb^j_{t_0}\Vert_{\xb^j_{t_0}}^2 + \zeta_j^2\right).
\end{align}
By using the triangle inequality and \eqref{eq:inexact_sol2} we deduce 
\begin{equation*}
\Vert\bar{\sb}^j_{t_0} - \xb^j_{t_0}\Vert_{\xb^j_{t_0}} \geq \Vert\sb^j_{t_0} - \xb^j_{t_0}\Vert_{\xb^j_{t_0}} - \Vert\sb^j_{t_0} - \bar{\sb}^j_{t_0}\Vert_{\xb^j_{t_0}}  \geq \zeta_j - \delta^j_0.
\end{equation*}
Hence, with $\delta^j_0 \leq \zeta_j$, this inequality implies
\begin{equation}\label{eq:lm43_est5}
\Vert\bar{\sb}^j_{t_0} - \xb^j_{t_0}\Vert_{\xb^j_{t_0}}^2 \geq \zeta_j^2 + (\delta^j_0)^2 - 2\zeta_j\delta^j_0.
\end{equation}
Combining \eqref{eq:lm43_est1}, \eqref{eq:lm43_est4} and \eqref{eq:lm43_est5}, we finally get
\begin{align}\label{eq:lm43_est6}
F(\xb^{j+1}_{t_0}; t_0) - F(\xb^j_{t_0}; t_0) \leq \omega_{*}(\alpha_j\zeta_j) - \zeta_j(\zeta_j - \delta_0^j)\alpha_j,
\end{align}
provided that $\alpha_j\zeta_j < 1$ and $\delta^j_0 \leq \zeta_j$.

Now we consider the function $\varphi(\alpha) := \zeta_j(\zeta_j - \delta_0^j)\alpha - \omega_{*}(\zeta_j\alpha)$. This function is concave, it attains the maximum at $\alpha_j := \frac{\zeta_j - \delta^j_0}{\zeta_j(1 + \zeta_j - \delta^j_0)}$ provided that $\delta^j_0 \leq \zeta_j$. In this case, we also have $\alpha_j\zeta_j = \frac{\zeta_j - \delta^j_0}{1 + \zeta_j - \delta^j_0} < 1$ and $\varphi(\alpha_j) = \omega(\zeta_j - \delta^j_0)$. Substituting this value into \eqref{eq:lm43_est6}, we obtain \eqref{eq:damped_PN_scheme}. %and then subtracting the result to $F(\xb^{*}_{t_0}; t_0)$ 
%\end{proof}
\Eproof
% End of the proof.

%%% Proof of Lemma 2.1.
\subsection{The proof of Lemma \ref{le:barrier_solution}}
%For simplicity of notation, we drop the subindex $t_k$ in $\xb^k_{t_k}$ as $\xb^k := \xb^k_{t_k}$.
Since $f$ is the barrier function of $\Omega$ and $\xb^{*}_t$ is the solution of \eqref{eq:cvx_prob}, it is obvious that
$\xb^{*}_t\in\mathrm{int}(\Omega)$ and $g(\xb^{*}) \leq g(\xb^{*}_t)$. 
We first prove \eqref{eq:approx_sol_P_and_CP}.
From \cite[Theorem 4.2.4]{Nesterov2004} we have 
\begin{equation}\label{eq:property_barrier1}
\nabla{f}(\xb)^T(\yb - \xb) < \nu, ~~\forall \xb, \yb \in \dom{f}.
\end{equation}
By using the convexity of $g$, the optimality condition \eqref{eq:optimality_for_ln_barrier} and the property \eqref{eq:property_barrier1} of the barrier function $f$, for any $\xb\in\dom{F} \equiv \dom{f}\cap\dom{g}$, we have
\begin{align}\label{eq:lm51_est1}
g(\xb) - g(\xb^{*}_t) &\geq (\xi^{*}_t)^T(\xb - \xb^{*}_t), ~~\xi^{*}_t \in\partial{g}(\xb^{*}_t)  \nonumber\\
&\overset{\tiny\eqref{eq:optimality_for_ln_barrier}}{\geq} - t\nabla{f}(\xb^{*}_t)^T(\xb - \xb^{*}_t)\\
&\overset{\tiny\eqref{eq:property_barrier1}}{\geq} -t\nu. \nonumber
\end{align}
By substituting $\xb = \xb^{*}$ in \eqref{eq:lm51_est1} we obtain \eqref{eq:approx_sol_P_and_CP}.
Since $\xb_{t_k} \in \Omega$ and $\xb^{*}_t$ is the optimal solution of 
Similarly, by letting $t = t_k$ and $\xb = \xb_{t_k}$ in \eqref{eq:lm51_est1} we obtain the left-hand side of \eqref{eq:rel_Fk_and_Fstar}.

Next, we prove the right-hand side of \eqref{eq:rel_Fk_and_Fstar}. 
By using \eqref{eq:approx_sol} in Definition \ref{de:inexact_sol} we can estimate
\begin{align}\label{eq:lm51_est2}
g(\xb_{t_k}) &\leq g(\bar{\xb}_{t_k}) + t_k\left[Q(\bar{\xb}_{t_k}; \xb_{t_{k-1}}) - Q(\xb_{t_k}; \xb_{t_{k-1}})\right] + t_k\frac{\delta_k^2}{2}\nonumber\\
&\leq g(\bar{\xb}_{t_k}) + t_k\nabla{f}(\xb_{t_{k-1}})^T(\bar{\xb}_{t_k} - \xb_{t_k}) + t_k\frac{\delta_k^2}{2}\nonumber\\
&+ \frac{t_{k}}{2}\left[\norm{\bar{\xb}_{t_k}-\xb_{t_{k-1}}}^2_{\xb_{t_{k-1}}} - \norm{\xb_{t_k} - \xb_{t_{k-1}}}^2_{\xb_{t_{k-1}}}\right],
\end{align}
where $\bar{\xb}_{t_k}$ is the exact solution of \eqref{eq:barrier_cvx_subprob1} at $t = t_k$.
Moreover, from the optimality condition \eqref{eq:opt_cond_subprob}, there exists $\bar{\vb}_{k}\in\partial{g}(\bar{\xb}_{t_k})$ such that 
\begin{equation}\label{eq:lm31_proof1}
\bar{\vb}_{k} = - t_k\nabla{f}(\xb_{t_{k-1}}) - t_k\nabla^2{f}(\xb_{t_{k-1}})(\bar{\xb}_{t_k} - \xb_{t_{k-1}}).
\end{equation}
By using the convexity of $g$ we can estimate $g(\xb^{*}_{t_k}) - g(\xb_{t_k})$ as
\begin{align}\label{eq:lm51_est3}
g(\xb^{*}_{t_k}) - g(\bar{\xb}_{t_k}) &\geq  \bar{\vb}_{k}^T(\xb^{*}_{t_k} -  \bar{\xb}_{t_k})  \nonumber\\
&\overset{\tiny\eqref{eq:lm31_proof1}}{=} -t_k\nabla{f}(\xb_{t_{k-1}})^T(\xb^{*}_{t_k} - \bar{\xb}_{t_k}) \nonumber\\
&- t_k(\bar{\xb}_{t_k} - \xb_{t_{k-1}})^T\nabla^2{f}(\xb_{t_{k-1}})(\xb^{*}_{t_k} - \bar{\xb}_{t_k}).
\end{align}
Now we sum up \eqref{eq:lm51_est2} and \eqref{eq:lm51_est3} and then rearrange the result by using the Cauchy-Schwarz inequality to get
\begin{align}\label{eq:lm51_est4}
g(\xb^{*}_{t_k}) - g(\xb_{t_k}) & \geq -t_k\nabla{f}(\xb_{t_{k-1}})^T(\xb^{*}_{t_k} - \xb_{t_k})  - \frac{t_k}{2}\delta_k^2 \nonumber\\
&-\frac{t_k}{2}\Big[ \Vert \bar{\xb}_{t_k} - \xb_{t_{k-1}}\Vert_{\xb_{t_{k-1}}}^2 - \Vert \xb_{t_k} - \xb_{t_{k-1}}\Vert_{\xb_{t_{k-1}}}^2 \nonumber\\
&+ 2(\bar{\xb}_{t_k} - \xb_{t_{k-1}})^T\nabla^2f(\xb_{t_{k-1}})(\xb^{*}_{t_k} - \bar{\xb}_{t_k})\Big]_{[1]}.
\end{align}
From \cite[Theorem 4.1.6]{Nesterov2004} we have
\begin{equation}\label{eq:lm51_est4c}
(1 - \lambda_{t_k})^2\nabla^2f(\xb^{*}_{t_k}) \preceq \nabla^2{f}(\xb_{t_{k-1}}) \preceq (1 - \lambda_{t_k})^{-2}\nabla^2{f}(\xb^{*}_{t_k}),
\end{equation}
where $\lambda_{t_k} := \Vert\xb_{t_{k-1}} - \xb^{*}_{t_k}\Vert_{\xb^{*}_{t_k}}$ defined as before. 
We can easily show that 
\begin{equation*}
\norm{\nabla{f}(\xb_{t_{k-1}})}_{\xb^{*}_{t_k}}^{*} \leq (1-\lambda_{t_k})^{-1}\norm{\nabla{f}(\xb_{t_{k-1}})}_{\xb_{t_{k-1}}}^{*} \overset{\tiny\eqref{eq:used}}{\leq} (1-\lambda_{t_k})^{-1}\sqrt{\nu}.
\end{equation*} 
Using this inequality together with the Cauchy-Shwarz inequality, we can prove that
\begin{align}\label{eq:lm51_proof3}
\nabla{f}(\xb_{t_{k-1}})^T(\xb^{*}_{t_k} - \xb_{t_k}) &\leq  \norm{\nabla{f}(\xb_{t_{k-1}})}_{\xb^{*}_{t_k}}^{*}\norm{\xb_{t_k} - \xb^{*}_{t_k}}_{\xb^{*}_{t_k}} \nonumber\\
& = \sqrt{\nu}(1- \lambda_{t_k})^{-1}\lambda_{t_k}^+.
\end{align} 
Next, we estimate the last term $[\cdots]_{[1]}$ of \eqref{eq:lm51_est4} as follows
\begin{align}\label{eq:lm51_est5}
[\cdots]_{[1]} & := \Vert \bar{\xb}_{t_k}  -  \xb_{t_{k-1}}\Vert_{\xb_{t_{k-1}}}^2  -  \Vert \xb_{t_k}  -  \xb_{t_{k-1}}\Vert_{\xb_{t_{k-1}}}^2  \nonumber \\ &+  2(\bar{\xb}_{t_k}  -  \xb_{t_{k-1}})^T\nabla^2f(\xb_{t_{k-1}})(\xb^{*}_{t_k}  -  \bar{\xb}_{t_k}) \nonumber\\
& = - \norm{\xb_{t_k} - \xb_{t_{k-1}}}_{\xb_{t_{k-1}}}^2  -  \norm{\bar{\xb}_{t_k} - \xb_{t_{k-1}}}_{\xb_{t_{k-1}}}^2 \nonumber \\ &+ 2(\bar{\xb}_{t_k} - \xb_{t_{k-1}})^T\nabla^2f(\xb_{t_{k-1}})(\xb^{*}_{t_k} - \xb_{t_{k-1}}) \nonumber\\
&\leq - \frac{1}{2}\norm{\bar{\xb}_{t_k} - \xb_{t_k}}_{\xb_{t_{k-1}}}^2 + 2(\bar{\xb}_{t_k} - \xb_{t_{k-1}})^T\nabla^2f(\xb_{t_{k-1}})(\xb^{*}_{t_k} - \xb_{t_{k-1}})\nonumber\\
&\overset{\tiny\eqref{eq:lm51_est4c}}{\leq}  2(1-\lambda_{t_k})^{-2}\left(\norm{\bar{\xb}_{t_k} - \xb_{t_k}}_{\xb_{t_{k-1}}} + \lambda_{t_k}^+ + \lambda_{t_k}\right)\lambda_{t_k} \nonumber\\
&\leq  2(1-\lambda_{t_k})^{-2}\left(\delta_k + \lambda_{t_k}^+ + \lambda_{t_k}\right)\lambda_{t_k}.
\end{align}
Here, the two last inequalities are obtained by using the triangle inequality, the definition of $\lambda_{t_k}^+$, $\lambda_{t_k}$ and \eqref{eq:computable_criterion2}.
Now, we combine \eqref{eq:lm51_est3}, \eqref{eq:lm51_est4} and \eqref{eq:lm51_est5} to derive
\begin{align*}
g(\xb^{*}_{t_k}) - g(\xb_{t_k}) &\geq -t_k\left[\sqrt{\nu}\frac{\lambda_{t_k}^+}{1 - \lambda_{t_k}} + (1-\lambda_{t_k})^{-2}\lambda_{t_k}\left(\lambda_{t_k}^+ + \lambda_{t_k} + \delta_k\right) + \frac{\delta_k^2}{2}\right],
\end{align*}
which is the right-hand side of \eqref{eq:rel_Fk_and_Fstar} provided that $ \lambda_{t_k} < 1$.
Finally, the estimate \eqref{eq:app_sol} follows directly by summing up \eqref{eq:approx_sol_P_and_CP} and \eqref{eq:rel_Fk_and_Fstar}.
Here, the left-hand side of \eqref{eq:app_sol} follows from the fact that $\xb_{t_k}\in\Omega$, and therefore, $g(\xb_{t_k}) - g(\xb^{*}) \geq 0$.
\Eproof
%% End of the proof.

%+ References.
\bibliographystyle{acm}

\end{document}